\documentclass{article}
\usepackage[dvipsnames, svgnames, x11names]{xcolor}
\usepackage{graphicx}
\usepackage{amsmath}
\usepackage{amssymb}
\usepackage{mathrsfs}
\usepackage[numbers,sort&compress]{natbib}
\usepackage{amsthm}
\usepackage{graphicx}
\usepackage{tikz}
\usepackage{stmaryrd}
\usepackage{graphicx}
\usepackage{subfigure}
\usepackage{hyperref}
\hypersetup{hypertex=true,
            colorlinks=true,
            linkcolor=blue,
            anchorcolor=blue,
            citecolor=red}
\newtheorem{theorem}{Theorem}[section]

\newtheorem{proposition}{Proposition}[section]
\newtheorem{RHP}{RHP}[section]

\DeclareMathOperator*{\im}{Im}
\DeclareMathOperator*{\re}{Re}
\numberwithin{equation}{section}
{\LARGE }

\begin{document}
		
	\baselineskip=16pt
	
	\title{ The complex  mKdV equation with step-like  initial data: Large time  asymptotic analysis    }
	\author{Zhaoyu Wang,  Kai Xu and   Engui Fan\thanks{Corresponding author: faneg@fudan.edu.cn}\\[4pt]
		School of Mathematical Science, Fudan University,\\
		Shanghai 200433, P.R. China}
	\maketitle

	\begin{abstract}
	  In this paper, we study large-time asymptotics  for the complex modified Korteveg-de Vries equation
\begin{equation}
u_t + \frac{1}{2}u_{xxx}+3|u|^2 u_x=0,
\end{equation}
with the step-like initial data
\begin{equation}
u(x,0)=u_0(x)=
\begin{cases}
0,  & {x \ge 0,}\\
A e^{iBx},   &{x < 0.}
\end{cases}
\end{equation}
It is shown that the  step-like initial problem can be described by a matrix  Riemann-Hilbert problem.
 We apply the steepest descent method to obtain different  large-time asymptotics in the
 the Zakharov-Manakov region, a  plane wave region and a slow decay region.
	\\[5pt]

	\end{abstract}

\newpage

\tableofcontents%

\section{Introduction}
Initial value problems for nonlinear integrable equations with the step-like initial data have aroused widespread concern since 1970s \cite{C1}-\cite{C4}. This paper focuses on the long-term behavior of solutions to such problems.   With the development of the theory of Whitham deformations \cite{C5} and the analysis of  matrix Riemann-Hilbert problem (RHP) representations of solutions of initial value problems \cite{C6}-\cite{C8}, the corresponding research  achieved a considerable progress.
%This method gave the most complete results in the case of integrable equations where the Lax pair were self-adjoint.
 Bikbaev discussed the case of the focusing nonlinear Schr$\ddot{\text{o}}$dinger (NLS) equation, where a much more complicated complex form of the theory of Whitham deformations was required \cite{C7}.

Deift and Zhou introduced the steepest descent method to deal with the asymptotics of solutions of integrable nonlinear equations \cite{C9}. Firstly, this method was widely applied for studying initial value problems with the  decaying initial data. Recently, Buckingham and Venakides used this approach to solve the problem with shock-type oscillating initial data by introducing the $g$-function.
  The purpose of introducing the $g$-function  is to transform  the original RHP into a form that can be treated asymptotically with the help of the related singular integral equations \cite{C10}. Moreover,  initial problems with non-zero boundary conditions were considered in references \cite{N1}-\cite{N3} for the defocusing NLS equation and in references \cite{N4}-\cite{N6} for the focusing NLS equation.

In this paper, we study the Cauchy problem for the complex modified Korteveg-de Vries(mKdV) equation
\begin{equation}
u_t + \frac{1}{2}u_{xxx}+3|u|^2 u_x=0, \label{cmkdv}
\end{equation}
with the step-like initial data
\begin{equation}\label{sl}
u(x,0)=u_0(x)=
\begin{cases}
0,  & {x \ge 0,}\\
A e^{iBx},   &{x < 0,}
\end{cases}
\end{equation}
where  $A>0$ and $B\in \mathbb{R}$ are real numbers.
The complex mKdV equation is a typical integrable partial differential equation, which is applicable to many physical phenomena, such as the propagation of  transverse waves in molecular chain model \cite{A1},
the transmission of electromagnetic waves in liquid crystal waveguide \cite{A2}-\cite{A4},  plane wave propagation in weakly nonlinear micropylar solid \cite{A5}, and so on.

The complex mKdV  equation is closely related to the cubic NLS equation:
\begin{equation}
iu_t + u_{xxx}+2|u|^2 u=0.
\end{equation}
Thus, as a member of the NLS hierarchy, the complex mKdV can be studied by the Darboux transformation, the Hirota method, the inverse scattering transformation(IST), the RHP, the steepest descent method, and so on.
We list some studies on the complex mKdV equation which are closed to our consideration.
The soliton solutions of the complex mKdV equation were obtained by tanh and sine-cosine method in \cite{M1} and \cite{M2}, the rogue wave solutions were derived in \cite{M3} and \cite{M4}, and the breathers and super-regular breathers were considered in \cite{M5} and \cite{M6}.
\cite{M7} and \cite{M8} discussed the existence and stability of solitary wave solutions and periodic traveling wave solutions of the complex mKdV equation. Conservation laws and exact group invariant solutions were constructed in \cite{M9}. Additionally, the discrete and nonlocal complex mKdV equations were given in \cite{M10} and \cite{M11} respectively, in which discrete solutions and nonlocal solitons were discussed.

Recently,  Zhang et al solved the complex mKdV equation when the reflection coefficient has multiple higher-order poles,  using the standard IST to establish an appropriate RHP with the zero boundary condition \cite{CT1}
\begin{equation*}
   u(x,0) = u_0(x) \to 0, \quad \text{as}\; x \to \infty.
\end{equation*}
Here  we discuss the large-time asymptotics of the complex mKdV equation with the non-zero boundary condition
\begin{equation*}
  u(x,0) = u_0(x) = A e^{iBx}  \nrightarrow 0, \quad x <0,
\end{equation*}
 which is different from the previous consideration.
The asymptotic behavior of the solution of the initial-value problem for the mKdV equation with the pure step-like initial data
\begin{equation*}
  u(x,0)= \begin{cases}
   0,  & {x \ge 0,}\\
   c,   &{x < 0,}
  \end{cases}
\end{equation*}
was studied in \cite{VA1}.
According to the sign of the function $\im \theta$ and the decay property of the jump matrix, they  divided the $(x,t)$ half-plane into three different regions, which included two zero-genus regions and one one-genus region.
Compared to the results of the  mKdV equation, we have an additional Zakharov-Manakov (ZM) region with zero genus, so there are three zero-genus regions for the complex mKdV equation.
Of particular interest to us is the work \cite{CT2}, where Monvel derived  the long-time asymptotic formulas of the focusing NLS equation with the step-like initial data under the form
\begin{equation*}
u(x,0)= \begin{cases}
0,  & {x \le 0,}\\
A e^{-2iBx},   &{x > 0.}
\end{cases}
\end{equation*}
They showed that there are three regions in the $(x,t)$ half-plane with two zero-genus regions and one one-genus region.
Compared to the results of the  focusing NLS equation,  our result has another slow decay region with zero genus.
 % The  initial-value  problem with the similar initial data for the defocuing NLS equation was be considered in \cite{CT3}.
% Other initial-value problems of the focusing and decusing NLS equation with different step-like oscillating background were considered in the following articles \cite{CT4}-\cite{CT7}.

Additionally, we find the complex mKdV equation (\ref{cmkdv}) has a plane wave solution
\begin{equation}\label{up}
u^p(x,t)=A e^{i(Bx+Ct)},
\end{equation}
with $C:=B^3/2-3A^2B$, which is consistent with $u^p(x,0)=u_0(x)$ for $x<0$. Then, we assume that the solution $u(x,t)$ of the initial problem (\ref{cmkdv})-(\ref{sl}) has the asymptotic behavior as follows
\begin{align}
u(x,t)&=o(1),\quad x \to +\infty, \\
u(x,t)&=u^p(x,t)+o(1),\quad x \to -\infty,
\end{align}
for any $t>0$.

The structure of this work is as follows.  In Section 2, we  give  results of
the spectral analysis of the Lax pair and  some important  properties of the scattering data and the reflection coefficient.
In Section 3, we formulate an RHP for  $M(z)$ to characterize the Cauchy problem  (\ref{cmkdv}) with the step-like initial data (\ref{sl}) and give our main theorem.
In Section 4,    we  focus on the long-time asymptotic analysis of this RHP under different regions, which are classified by the relation between $x$ and $t$.

 \section{Spectral analysis on  the Lax pair}
It is well-known that the Lax pair  of the complex mKdV equation can be given by
\begin{equation}\label{lax1}
    \begin{cases}
    \Phi_x+ik\sigma_3 = U \Phi,\\
    \Phi_t+2ik^3\sigma_3= V \Phi,
    \end{cases}
\end{equation}
with

  \begin{align*}
  \sigma_3&:=\left(\begin{array}{cc}
  1 & 0\\
  0 & -1
  \end{array}\right), \quad U:=\left(\begin{array}{cc}
  0 & u(x,t)\\
  -\bar{u}(x,t) & 0
  \end{array}\right),\\
  V&= -2kU^2\sigma_3 +\frac{1}{2} \left(U_x U - U U_x\right) \sigma_3 + 2k^2 U +ik\sigma_3 U_x -\frac{1}{2}U_{xx}+U^3,
  \end{align*}
where $k$ is a spectral parameter.

The corresponding Lax pair for  the plane wave solution (\ref{up}) is given by
%A particular solution of the initial problem (\ref{cmkdv})-(\ref{sl}) with the plane wave solution (\ref{up})
\begin{equation}
\begin{cases}
\Phi^p_x+ik\sigma_3 = U^p \Phi^p,\\
\Phi^p_t+2ik^3\sigma_3= V^p \Phi^p,
\end{cases}
\end{equation}
with
\begin{align}
U^p&=\left(\begin{array}{cc}
0 &  A e^{i(Bx+Ct)}\\
-A e^{-i(Bx+Ct)} & 0
\end{array}\right),\\
V^p&=-2k(U^p)^2\sigma_3 +\frac{1}{2} \left(U^p_x U^p - U^p U^p_x\right) \sigma_3 + 2k^2 U^p +ik\sigma_3 U^p_x -\frac{1}{2}U^p_{xx}+(U^p)^3.
\end{align}
Then, we choose a particular solution $\Phi^p(k;x,t)$ in the form
\begin{equation}
\Phi^p(k;x,t)= e^{i(Bx+Ct)\sigma_3/2} E_0(k) e^{-i(xX(k)+t\Omega(k))\sigma_3},
\end{equation}
where
\begin{align}
E_0(k)=\frac{1}{2} \left(\begin{array}{cc}
\varphi(k)+\frac{1}{ \varphi(k)} &  \varphi(k)-\frac{1}{ \varphi(k)}\\
\varphi(k)-\frac{1}{ \varphi(k)} & \varphi(k)+\frac{1}{ \varphi(k)}
\end{array}\right), \quad \varphi(k)=\left( \frac{k-iA+\frac{B}{2}}{k+iA+\frac{B}{2}}\right)^{\frac{1}{4}},\label{var}\\
X(k)=\sqrt{\left(k+\frac{B}{2}\right)^2+A^2}, \quad \Omega(k)=\left(2k^2-Bk+\frac{1}{2}B^2-A^2\right) X(k).\label{XO}
\end{align}
In addition,  the functions $X(k)$ and $\Omega(k)$ are analytic in the region $\mathbb{C} 	\backslash [-B/2+iA, -B/2-iA ]$.
We have the asymptotics of $X(k)$ and $\Omega(k)$: as $k \to \infty$,
\begin{align}
X(k)=k+\frac{B}{2}+\mathcal{O}\left(\frac{1}{k}\right),\; \Omega(k)=2k^3+\frac{B^3-6A^2B}{4}-\frac{3(A^4-A^2B^2)}{4k}+\mathcal{O}\left(\frac{1}{k^2}\right).
\end{align}
%The branch cut for $X$ and $\Omega$ is taken along the segment $\gamma \cup \bar{\gamma}$ with $\gamma=\{ k=k_1+ik_2 \in \mathbb{C}|k_1=-B/2,\;k_2 \le A \}$.

Define the following Jost solutions of the Lax equations (\ref{lax1}) for $\im k=0$,
\begin{align}\label{jost1}
 \Phi(k;x,t)&=e^{-i(kx+2k^3t)\sigma_3} + o(1), \quad x\to +\infty,\\
 \Psi(k;x,t)&=\Phi^p(k;x,t)+ o(1), \quad x\to -\infty.\label{jost2}
\end{align}
% $E(k;x,t)$ is the solution of the linear differential equations in here,
%\begin{align*}
%  E_x = M^- E,\quad E_t= N^- E,
%\end{align*}
%with
%\begin{align*}
%& M^-=-ik\sigma_3+ U^-,\quad U^-=\left(\begin{array}{cc}
% 0 &  A e^{i(Bx+Ct)}\\
%  -A e^{-i(Bx+Ct)} & 0
% \end{array}\right),\\
%& N^-=-2ik^3 \sigma_3 -2k(U^-)^2\sigma_3 +\frac{1}{2} \left(U^-_x U^- - U^- U^-_x\right) \sigma_3 + 2k^2 U^- +ik\sigma_3 U^-_x -\frac{1}{2}U^-_{xx}+(U^-)^3.
%\end{align*}
%Then, we can chose a particular solution $E(k;x,t)$
%\begin{equation}
%E(k;x,t)= e^{i(Bx+Ct)\sigma_3/2} E_0(k) e^{-i(xX(k)+t\Omega(k))\sigma_3}
%\end{equation}
%where
%\begin{align}
%  E_0(k)=\frac{1}{2} \left(\begin{array}{cc}
%  \varphi(k)+\frac{1}{ \varphi(k)} &  \varphi(k)-\frac{1}{ \varphi(k)}\\
%  \varphi(k)-\frac{1}{ \varphi(k)} & \varphi(k)+\frac{1}{ \varphi(k)}
%  \end{array}\right), \quad \varphi(k)=\left( %\frac{k-iA+\frac{B}{2}}{k+iA+\frac{B}{2}}\right)^{\frac{1}{4}},\\
%  X(k)=\sqrt{\left(k+\frac{B}{2}\right)^2+A^2}, \quad \Omega(k)=\left(2k^2-Bk+\frac{1}{2}B^2-A^2\right) X(k).
%\end{align}
The solutions of (\ref{jost1}) and (\ref{jost2}) can be represented in the following form
\begin{align}\label{phi1}
    \Phi(k;x,t)&=e^{-i(kx+2k^3t)\sigma_3}+ \int_{+\infty}^{x} e^{ik(x'-x)\sigma_3} U(x',t)  \Phi(k;x',t) \mathrm{d} x',\\
    \Psi(k;x,t)&=\Phi^p(k;x,t)+ \int_{-\infty}^{x} \Phi^p(k;x,t)\left(\Phi^p\right)^{-1}(k;x',t) \left(U(x',t)-U^-(x',t) \right) \Psi(k;x',t) \mathrm{d} x'. \label{psi1}
\end{align}
For convenience, denote $E=-B/2+iA$, then $\bar{E}=-B/2-iA$.
\begin{proposition} \label{pro1}
	The solutions $\Phi(k;x,t)$ and $\Psi(k;x,t)$ have the following properties:
\begin{itemize}
	\item[(1)]  Determinant:
	\begin{equation*}
	\det \Phi(k;x,t)=1,\quad \det  \Psi(k;x,t)=1.
	\end{equation*}
	\item[(2)] Analyticity: \\
	$\Phi_1(k;x,t)$ is analytic in $\mathbb{C}^-$ and continuous up to the boundary,\\
	 $\Phi_2(k;x,t)$ is analytic in $\mathbb{C}^+$ and continuous up to the boundary,\\
	$\Psi_1(k;x,t)$ is analytic in $\mathbb{C}^+	\backslash [E, \bar{E} ]$ and continuous up to the boundary,\\
	$\Psi_2(k;x,t)$ is analytic in $\mathbb{C}^-	\backslash [E, \bar{E}]$ and continuous up to the boundary.
	\item[(3)] Symmetry:
	\begin{equation}
	    \Phi(k;x,t)=\sigma_2  \overline{\Phi(\bar{k};x,t)} \sigma_2,\quad \Psi(k;x,t)=\sigma_2  \overline{\Psi(\bar{k};x,t)} \sigma_2,
	    	\end{equation}
	    where $\sigma_2=\left(\begin{array}{cc}
	    0 & i\\
	    -i& 0
	    \end{array}\right)$.
 	\item[(4)] Large $k$ asymptotics:
 \begin{align*}
 &\Phi_1(k;x,t) e^{i(kx+2k^3t)}= \left(\begin{array}{cc}
 1\\
  0
 \end{array}\right)+ \mathcal{O}\left(\frac{1}{k}\right),\quad k \to \infty, \quad \im k \le 0,\\
 &\Phi_2(k;x,t) e^{-i(kx+2k^3t)}= \left(\begin{array}{cc}
 0\\
 1
 \end{array}\right)+ \mathcal{O}\left(\frac{1}{k}\right),\quad k \to \infty, \quad \im k \ge 0,\\
  &\Psi_1(k;x,t) e^{i(kx+2k^3t)}= \left(\begin{array}{cc}
 1\\
 0
 \end{array}\right)+ \mathcal{O}\left(\frac{1}{k}\right),\quad k \to \infty, \quad \im k \ge 0,\\
 &\Psi_2(k;x,t) e^{-i(kx+2k^3t)}= \left(\begin{array}{cc}
 0\\
 1
 \end{array}\right)+ \mathcal{O}\left(\frac{1}{k}\right),\quad k \to \infty, \quad \im k \le 0.
 \end{align*}
\end{itemize}
\end{proposition}

Since the eigenfunctions $\Phi(k;x,t)$ and $\Psi(k;x,t)$  are the solutions of the Lax pair (\ref{lax1}), there exists a matrix $S(k)$ which is independent of $x$ and $t$
\begin{equation}\label{S}
    S(k)= \left(\Phi(k;x,t)\right)^{-1} \Psi(k;x,t).
\end{equation}
We find
\begin{equation}
    S(k)=\left(\Phi(0,0;k)\right)^{-1} \Psi(0,0;k)=E_0(k).
\end{equation}
Denote $a(k)=S_{11}(k)$ and $b(k)=S_{12}(k)$, where $S_{ij}$ represents the $i$th row and $j$th column of the matrix $S$. Then according to the symmetries of $\Phi(k;x,t)$ and $\Psi(k;x,t)$, we have
\begin{equation*}
S(k)= \left(\begin{array}{cc}
a(k) & b(k)\\
b(k) &a(k)
\end{array}\right),
\end{equation*}
where $a(k)=\bar{a}(\bar{k})=\frac{1}{2} \left( \varphi(k)+\frac{1}{ \varphi(k)} \right)$ and $b(k)=-\bar{b}(\bar{k})=\frac{1}{2} \left( \varphi(k)-\frac{1}{ \varphi(k)} \right)$.
 Define the reflection coefficient
\begin{equation*}
  r(k) := \frac{b(k)}{a(k)}.
\end{equation*}
\begin{proposition} \label{pro2}
	 $a(k)$, $b(k)$ and $r(k)$ have the following properties:
	\begin{itemize}
		\item [(1)] Analyticity: $a(k)$ is analytic in $\mathbb{C}^+ \backslash [E, -B/2]$ and has a continuous extension to $\left(E,-B/2 \right)_- \cup\left(E,-B/2 \right)_+ $.
		$b(k)$ is continuous in $k \in \mathbb{R} \cup \left(E,\bar{E} \right)$. The function $r(k)$ is defined in $\mathbb{R} \cup \left(E,\bar{E} \right)$ except for the points where $a(k)=0$.
		\item [(2)] Asymptotic behavior:
		\begin{align*}
		&a(k)=1+\mathcal{O}(k^{-1}), \quad k \to \infty, \;\im k \ge 0,\\
		&r(k)=\mathcal{O}(k^{-1}).
		\end{align*}
		\item[(3)]  For $k \in \left(E,\bar{E} \right) $, $a_-(k)=-i\overline{b_+(\bar{k})}$ and $a_+(k)=i\overline{b_-(\bar{k})}$.
        \item[(4)]   Let  \begin{equation}\label{f}
                       f(k):=\frac{i}{a_+(k)a_-(k)},\; k \in (E,-B/2),
                    \end{equation}
                    then we have
                    \begin{align*}
                        f(k)=r_-(k)-r_+(k), \; k \in (E,-B/2).
                    \end{align*}
    \end{itemize}

\end{proposition}
\begin{proof}
	Proofs of Propositions \ref{pro1} and \ref{pro2} are similar to the proof in Section 2  in \cite{xu}.
\end{proof}

\section{A RH problem with step-like initial data}
Rewrite the formula (\ref{S}) into  the vector form:
\begin{align*}
     &\frac{\Psi_1(k;x,t)}{a(k)}=\Phi_1(k;x,t) + r(k) \Phi_2(k;x,t),\\
     &\frac{\Psi_2(k;x,t)}{a(k)}=r(k)\Phi_1(k;x,t) + \Phi_2(k;x,t).
\end{align*}
Then we define the matrix $M(k;x,t)$ in the following way:
\begin{equation*}
    M(k;x,t)=\begin{cases}
        \left(\frac{\Psi_1(k;x,t)}{a(k)}e^{it\theta(k)}, \Phi_2(k;x,t) e^{-it\theta(k)}  \right), \; k\in \mathbb{C}^+ \backslash [E,-B/2],\\
        \left(\Phi_1(k;x,t) e^{it\theta(k)}, \frac{\Psi_2(k;x,t)}{\overline{a(\bar{k})}} e^{-it\theta(k)} \right), \; k\in \mathbb{C}^-\backslash [-B/2,\bar{E}],
    \end{cases}
\end{equation*}
where $\theta(k)=\theta(k;x,t)=2k^3+\frac{kx}{t}$.
Thus we obtain the following RHP of $M(k;x,t)$.

\begin{RHP}\label{RHP1}
Find a matrix-valued function $M(k;x,t)$ which satisfies
\begin{itemize}
    \item Analyticity: $M(k;x,t)$ is analytic in $\mathbb{C}\setminus \Sigma$, where $\Sigma=\mathbb{R}\cup \left(E,\bar{E} \right)$.
    \item Boundedness: $M(k;x,t)$ is bounded at the points $E$ and $\bar{E}$.
    \item Jump condition: $M(k;x,t)$ satisfies the jump condition
    $$M_-(k;x,t)=M_+(k;x,t)V(k;x,t), \; k \in \Sigma,$$
    where
    \begin{equation}\label{V0}
       V(k;x,t)=\begin{cases}
        \left(\begin{array}{cc}
            1 & -\overline{r(k)}     e^{-2it \theta(k)}\\
            -r(k)e^{2it \theta(k)} & 1+|r(k)|^2
        \end{array}\right),\; k \in \mathbb{R},\\
    \left(\begin{array}{cc}
        1 & 0\\
       f(k)e^{2it \theta(k)} & 1
    \end{array}\right),\; k \in [E,-B/2],\\
    \left(\begin{array}{cc}
        1 & - \overline{f(\bar{k})} e^{-2it \theta(k)}\\
       0 & 1
    \end{array}\right),\; k \in [-B/2,\bar{E}]
\end{cases}
\end{equation}
with $f(k)$ being given in (\ref{f}).

    \item Asymptotic behavior:
    \begin{align*}
            M(k;x,t)=I+\mathcal{O}(k^{-1}),	\quad  k \to  \infty.
    \end{align*}
\end{itemize}
\end{RHP}
The solution $u(x,t)$ of the initial value problem (\ref{cmkdv})-(\ref{sl}) can be reconstructed by the solution of RHP \ref{RHP1} as follows:
\begin{equation}
      u(x,t)=2i \lim_{k \to \infty} k\left(M(k;x,t)\right)_{12},
\end{equation}
where $\left(M(k;x,t)\right)_{12}$ is the $12$th entry of the solution $M(k;x,t)$ of  RHP \ref{RHP1}.

\section{Large-time asymptotics for complex mKdV equation }
\begin{figure}
	\begin{center}
		\subfigure[$\xi>0$]{
			\begin{tikzpicture}
			\node[anchor=south west,inner sep=0](image) at (7,0)
			{\centering\includegraphics[width=0.35\textwidth]{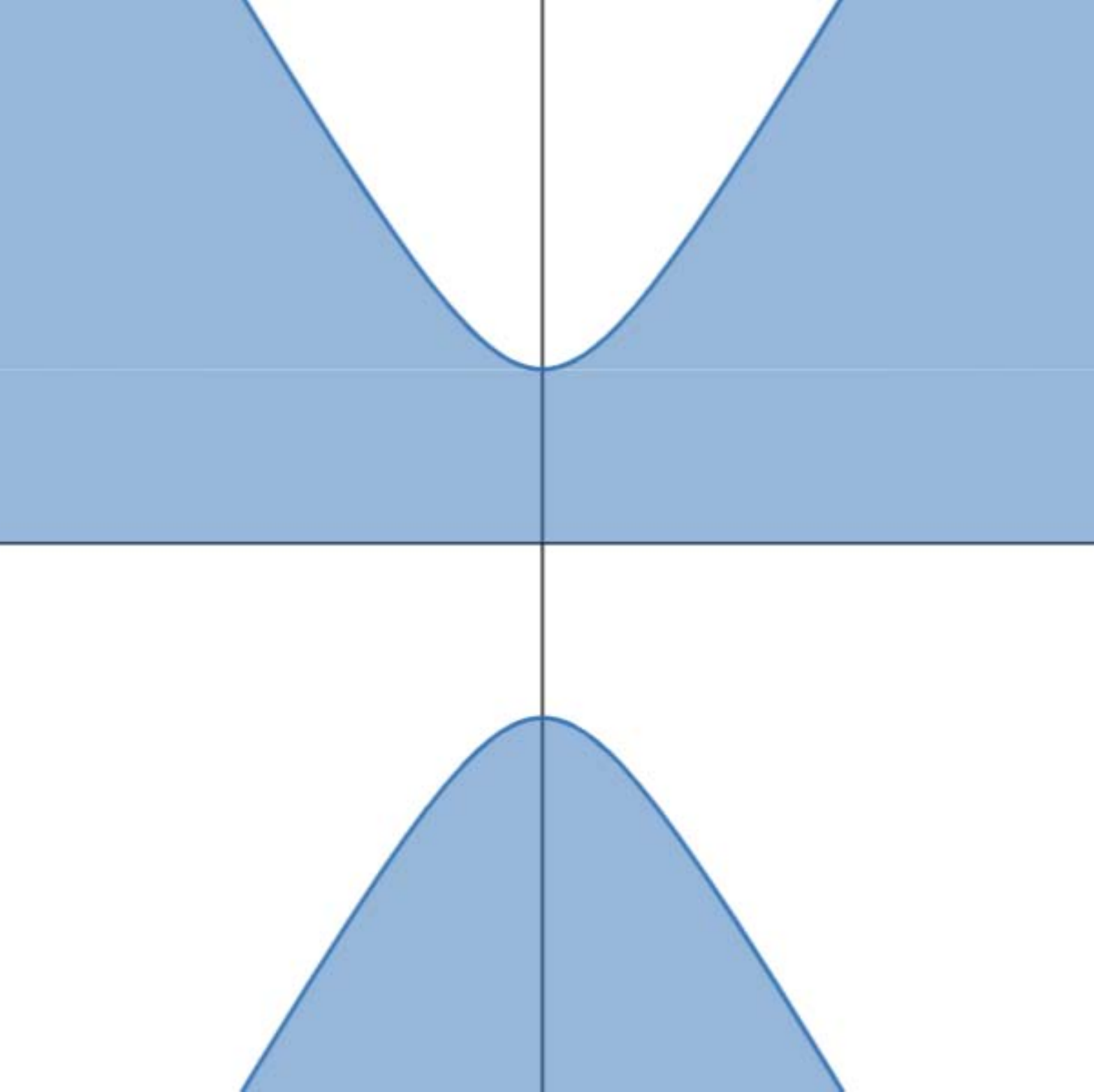}};
			\begin{scope}[x={(image.south east)},y={(image.north west)}]
			
			%	\node at (0.25,0.7) {\footnotesize$|e^{2 i t \theta}| \rightarrow 0$};
			%	\node at (0.52,0.7) {\footnotesize$|e^{2 i t \theta}| \rightarrow 0$};
			%	\node at (0.75,0.28) {\footnotesize$|e^{-2 i t \theta}| \rightarrow 0$};
			%    \node at (0.33,0.82) {\footnotesize$|e^{-2 i t \theta}| \rightarrow 0$};
			%	\node at (0.49,0.26) {\footnotesize$|e^{-2 i t \theta}| \rightarrow 0$};
			%	\node at (0.68,0.15) {\footnotesize$|e^{2 i t \theta}| \rightarrow 0$};

			\node at (0.25,0.7) {$+$};
			\node at (0.52,0.7) {$+$};
			\node at (0.75,0.28) {$-$};
			\node at (0.33,0.82) {$-$};
			\node at (0.49,0.26) {$-$};
			\node at (0.68,0.15) {$+$};

			\end{scope}		
			\end{tikzpicture}
		}%
		\quad\quad  \subfigure[$\xi<0$]{
			\begin{tikzpicture}
			\node[anchor=south west,inner sep=0](image) at (7,0)
			{\centering\includegraphics[width=0.35\textwidth]{FIG1.pdf}};
			\begin{scope}[x={(image.south east)},y={(image.north west)}]
			% ÓÒÉÏ1
			%		\node at (0.57,0.7) {\footnotesize$|e^{2 i t \theta}| \rightarrow 0$};
			% ÓÒÏÂ1
			%		\node at (0.78,0.32) {\footnotesize$|e^{-2 i t \theta}| \rightarrow 0$};
			% ×óÏÂ1
			%		\node at (0.42,0.32) {\footnotesize$|e^{-2 i t \theta}| \rightarrow 0$};
			% ÖÐÏÂ
			%		\node at (0.62,0.25) {\footnotesize$|e^{2 i t \theta}| \rightarrow 0$};
			% ×óÉÏ1
			%		\node at (0.22,0.7) {\footnotesize$|e^{2 i t \theta}| \rightarrow 0$};
			% ÖÐÉÏ
			%		\node at (0.38,0.72) {\footnotesize$|e^{-2 i t \theta}| \rightarrow 0$};
			
			\node at (0.57,0.7) {$+$};
			\node at (0.78,0.28) {$-$};
			
			\node at (0.43,0.28) {$-$};
			\node at (0.22,0.7) {$+$};
			
			\node at (0.68,0.15) {$+$};
			\node at (0.34,0.8) {$-$};

			\end{scope}		
			\end{tikzpicture}
		}%
	\end{center}
	\caption{\small The sign of $\im \theta(k)$: The signal $+$ stands for $\im \theta(k)>0$ in the blue region, while the signal $-$ stands for $\im \theta(k)<0$ in the white region.}
	\label{xio1}
\end{figure}

It is widely known that the jump matrix $V(k;x,t)$ in (\ref{V0}) depends on the exponential term $e^{\pm 2i \theta(k)}$.  Thus the  sign of $\im \theta(k)$  plays an important role in  analyzing the asymptotic behavior of the solution $M(k;x,t)$.
Assume $k=k_1+ik_2$ and introduce $\xi=\frac{x}{6t}$ and $\theta(k)$ satisfying $\theta(k)=2k^3+6k\xi$.
Then we obtain
\begin{align*}
\im \theta(k)=-2 k_2^3+ 6k_2(k_1^2+\xi).
\end{align*}
The sign of the $\im \theta(k)$ is depicted in Figure \ref{xio1}.

According to the sign of the function  $\im \theta(k)$ and the number of genus of $g(k;\xi)$,
 which will be given in the subsequent section, the regions of $\xi$ can be divided into four regions.
\begin{itemize}
	\item 	The  slow decay region  in the genus-$0$ sector:            $\xi \in \{\xi \in \mathbb{R}^+|\xi>\frac{A^2}{3}-\frac{B^2}{4}\}$.

   \item 	The Zakharov-Manakov(Z-M) region  in the genus-$0$ sector : $\xi \in \{\xi \in \mathbb{R}^-|\xi>\frac{A^2}{3}-\frac{B^2}{4}\}$.

	\item  The elliptic wave region  in the genus-$1$ sector: $\xi \in \{\xi \in \mathbb{R}^-|  -\frac{B^2}{4}<\xi < \frac{A^2}{3}-\frac{B^2}{4} \}$.
	
\item 	The plane wave region  in the genus-$0$ sector:  $\xi \in \{\xi \in \mathbb{R}^-|\xi<-\frac{B^2}{4}  \}$.

\end{itemize}
See Figure \ref{region}.
In this paper, we only consider the case of 0 genus and we will consider other cases of higher genus later. The asymptotic results in different regions are shown in the following theorem.

\begin{figure}
	\begin{center}
		\begin{tikzpicture}

		\draw[Blue!20,fill=Blue!15] (0,0)--(-4.8,4.5)--(-4.8,0);
		\node at (-3.5,1.4) {plane wave};
        \node at (-3.5,0.9) {genus-0};
        \draw[LightSteelBlue!20,fill=LightSteelBlue!10] (0,0)--(-4.8,4.5)--(-1.98,4.5);
	    \node at (-2.4,3.2) {elliptic wave};
        \node at (-2.1,2.6) {genus-1};
        \draw[Green!10,fill=Green!15] (0,0)--(0,4.5)--(4.8,4.5)--(4.8,0);
		\node at (2.5,2.4) {slow decay};
        \node at (2.5,1.9) {genus-0};
        \draw[Yellow!10,fill=Yellow!15] (0,0)--(-1.98,4.5)--(0,4.5);	
        \node at (-0.5,3.4) {Z-M};
        \node at (-0.6,2.9) {genus-0};
        \node at (-1.9,5) {$\xi = \frac{A^2}{3}-\frac{B^2}{4}$};
         \node at (-4.7,5) {$\xi = -\frac{B^2}{4}$};
		\draw[ ->](-5.5,0)--(5.5,0)node[black,right]{$x$};
		\draw[ ->](0,0)--(0,5.5)node[black,right]{$t$};
	  \draw[dashed,red](-4.8,4.5)--(0,0);
		\draw[dashed,red](-1.98,4.5)--(0,0);
		\end{tikzpicture}
	\end{center}
	\caption{ \small  The different asymptotic  regions of $(x,t)$-plane.}
	\label{region}
\end{figure}
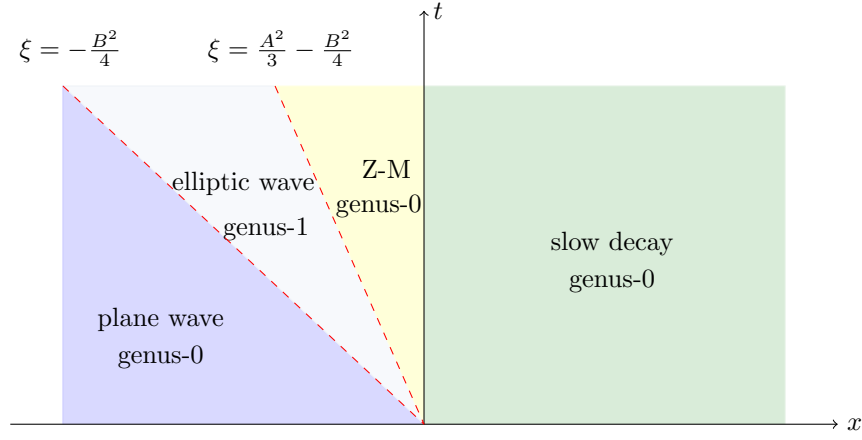

\begin{theorem} \label{Th1}
  The solution $u(x,t)$ of the initial problem (\ref{cmkdv})-(\ref{sl}) under the  step-like initial data can be constructed by the following formulas.

\begin{itemize}
	\item[$\blacktriangleright$ ] (Z-M region) For
 $\left(2A^2-\frac{3 }{2} B^2\right)t<x<0$, as $t\to+\infty$,
	% the solution $u(x,t)$ is
	\begin{equation}
	u(x,t)= \frac{v e^{\frac{\pi v}{2}}}{\sqrt{12tk_0 \pi}} \left(  \delta_{k_0}^2 e^{-\frac{3 \pi i}{4}} \Gamma(iv) \overline{r(k_0)} - \delta_{k_0}^{-2} e^{\frac{3 \pi i}{4}} \Gamma(-iv) r(k_0) \right) +
\mathcal{O}\left(t^{-1} \log t\right),
	\end{equation}
	where
	\begin{align*}
	& \delta_{k_0} = \left(24t(k_0)^3 \right)^{-\frac{iv}{2}} e ^{4it k_0^3} e^{\chi(k_0)}, \\
	&k_0 = \sqrt{-\frac{x}{6t}},\; v = -\frac{1}{2\pi} \log \left(1+|r(k_0)|^2\right),\\
	& \chi(k_0)=\frac{1}{2\pi i} \int_{-k_0}^{k_0} \frac{v(s)-v(k_0)}{s-k_0} \mathrm{d} s.
	\end{align*}
	\item[$\blacktriangleright$ ] (Plane wave region) For
  $ x<-\frac{3}{2}B^2t $, as $t \to +\infty$,
		%the solution $u(x,t)$ takes the form of a plane wave
		\begin{equation*}
		u(x,t)= A e^{i \left(Ct+Bx-2 \phi(\xi)\right)} + \mathcal{O}\left( t^{-1/2}\right), t \to \infty,
		\end{equation*}
		where $\phi(\xi)$ is given in (\ref{phi}).
	\item[$\blacktriangleright$ ] (Slow decay region) For 	
  $x>\left( 2A^2 -\frac{3}{2}B^2\right)t$ with $x>0$,  as  $t \to +\infty$,
         %the solution $u(x,t)$ can be estimated by the following formula as $t \to +\infty$,
		\begin{align}
		u(x,t) = \begin{cases} \mathcal{O}\left(  \frac{e^{12\sqrt{3}t   \left( \left(\xi +\frac{B^2}{4}\right)^{3/2} + \left( \xi +\frac{B^2}{4} \right)^{1/2} \right)}}{ \sqrt{t}}    \right),\quad 2 A^2 - \frac{3}{2}B^2 t < x< A^2 t,\\
		\mathcal{O} \left( \frac{e^{12\sqrt{6t}\xi^{3/2}}}{\sqrt{t}}\right),\quad A^2t <x< 6A^2 t,\\
		\mathcal{O} \left( \frac{e^{-8t \xi^{3/2}}}{\sqrt{t}}\right), \quad x>6A^2t.
		\end{cases}
		\end{align}
		%where $f(x,y) = 4y^3 -12\xi y -12 x^2y $ and $f_{max} = \max_{\sqrt{6\xi}< y <\sqrt{3} \left( \xi + \frac{B^2}{4} \right)}f(-B/2,y)$.

\end{itemize}
\end{theorem}

In the following sections, we will give proofs for each of these three regions.

\subsection{The Z-M region}

 We first consider the large-time  asymptotic behavior in the Z-M region $\xi>\frac{A^2}{3}-\frac{B^2}{4}$ under the condition $\xi<0$, which is $\frac{A^2}{3}-\frac{B^2}{4}<\xi<0$ with $A<\sqrt{\frac{3B^2}{4}}$.
Define $k_0=\sqrt{-\frac{x}{6t}}$. We  decompose the jump matrix $V(k;x,t)$ (\ref{V0}) on $\mathbb{R}$ into the form as follows:
\begin{align}
    V(k)=\begin{cases}\label{de}
        \left(\begin{array}{cc}
            1 & 0\\
            -r(k)e^{2it \theta(k)} & 1+|r(z)|^2
        \end{array}\right)  \left(\begin{array}{cc}
            1 & -\overline{r(k)}     e^{-2it \theta(k)}\\
            0 & 1
        \end{array}\right),\; k \in \left(k_0,+\infty \right) \cup \left(-\infty,-k_0\right),\\
        \left(\begin{array}{cc}
            1 & -\frac{\bar{r}}{1+|r|^2} e^{-2it \theta(k)} \\
            0 & 1
        \end{array}\right)  \left(\begin{array}{cc}
            \frac{1}{1+|r|^2} & 0\\
            0 & 1+|r|^2
        \end{array}\right)  \left(\begin{array}{cc}
            1 & 0\\
            -\frac{r}{1+|r|^2} e^{2it \theta(k)} & 1
        \end{array}\right),\; k \in \left(-k_0,k_0\right).
    \end{cases}
\end{align}
     To find the solution of  RHP \ref{RHP1}, we make the following three transformations:
     \begin{itemize}
        \item First transformation: Remove the middle matrix in the second decomposition in  (\ref{de}).
        \item Second transformation:  Open the jump line by analytic continuation.
        \item Third transformation: Match the parabolic cylinder model.
    \end{itemize}
  %  We make the first transformation by introducing a scalar RHP of $\delta(k)$, which is
    We first introduce a scalar RHP of $\delta(k)$:

\begin{RHP}
  Find a matrix-valued function $\delta(k)$ which satisfies
\begin{itemize}
    \item Analyticity: $\delta(k)$ is analytic in $\mathbb{C}\setminus [-k_0,k_0]$, where $\Sigma=\mathbb{R}\cup \left(E,\bar{E} \right)$.
    \item Jump condition: $\delta_-(k)=\delta_+(k) \left(1+|r|^2\right)^{-1}$, $|k|<k_0$.
    \item Asymptotic behavior: $\delta(k)\to1$, $k\to \infty$.
\end{itemize} 	
\end{RHP}

Using Plemelj formula, we find
\begin{equation*}
     \delta(k)=\exp \left(\frac{1}{2\pi i} \int_{-k_0}^{k_0} \frac{\log \left(1+|r(s)|^2\right)}{s-k} \mathrm{d} s\right) = \left(\frac{k-k_0}{k+k_0}\right)^{iv(k_0)} e^{\chi(k)},
\end{equation*}
where
\begin{equation}
    v(k)=-\frac{1}{2\pi} \log \left(1+|r(k)|^2\right),\; \chi(k)=\frac{1}{2\pi i} \int_{-k_0}^{k_0} \frac{v(s)-v(k_0)}{s-k} \mathrm{d} s.
\end{equation}
Then, we make the first transformation
\begin{equation}
       M^{(1)}(k;x,t)=M(k;x,t) \delta^{-\sigma_3}(k),
\end{equation}
which satisfies the following RHP:

\begin{RHP}
	 Find a matrix-valued function $M^{(1)}(k;x,t)$ which satisfies
\begin{itemize}
    \item Analyticity: $M^{(1)}(k;x,t)$ is analytic in $\mathbb{C}\setminus \Sigma^{(1)}$, where $\Sigma^{(1)}=\Sigma$.
    \item Jump condition: $M^{(1)}(k;x,t)$ satisfies the jump relation
    \begin{equation}\label{V1}
        V^{(1)}(k;x,t)=\begin{cases}
         \left(\begin{array}{cc}
             1 & -\frac{\overline{r(k)}}{1+|r(k)|^2}  \delta_+^2(k) e^{-2it \theta(k)}\\
             0 & 1
         \end{array}\right)    \left(\begin{array}{cc}
            1 & 0\\
            -\frac{r(k)}{1+|r(k)|^2}  \delta_-^{-2}(k) e^{2it \theta(k)} & 1
        \end{array}\right) ,\; |k|<k_0,\\
     \left(\begin{array}{cc}
         1 & 0\\
        -r(k) \delta^{-2}(k) e^{2it \theta(k)} & 1
     \end{array}\right)\left(\begin{array}{cc}
        1 & -\overline{r(\bar{k})} \delta^{2}(k) e^{-2it \theta(k)} \\
       0& 1
    \end{array}\right),\; |k|>k_0,\\
    \left(\begin{array}{cc}
        1 & 0\\
       f(k) \delta^{-2}(k)e^{2it \theta(k)} & 1
    \end{array}\right),\; k \in [E,-B/2],\\
     \left(\begin{array}{cc}
         1 & - \overline{f(\bar{k})} \delta^{2}(k)e^{-2it \theta(k)}\\
        0 & 1
     \end{array}\right),\; k \in [-B/2,\bar{E}].
 \end{cases}
 \end{equation}
    \item Asymptotic behavior: $M^{(1)}(k;x,t)\to I$, $k\to \infty$.
\end{itemize}
\end{RHP}

Next,  the jump matrix $V^{(1)}(k;x,t)$ can be analytically extended from the boundary $\Sigma^{(1)}$ by doing the second transformation
\begin{equation*}
     M^{(2)}(k;x,t)=M^{(1)}(k;x,t) G(k),
\end{equation*}
where
\begin{align}
    G(k)=\begin{cases}\label{G}
        \left(\begin{array}{cc}
            1 & 0\\
            -r(k)  \delta^{-2}(k)e^{2it \theta(k)} & 1
        \end{array}\right),\; k \in \Omega_{11} \cup \Omega_{21},\\
       \left(\begin{array}{cc}
            1 & -\overline{r(\bar{k})} \delta^{2}(k) e^{-2it \theta(k)} \\
           0 & 1
        \end{array}\right),\; k\in \Omega_{14} \cup \Omega_{24},\\
        \left(\begin{array}{cc}
            1 & -\frac{\overline{r(\bar{k})}}{1+|r(k)|^2} \delta_+^2(k)^{-2it \theta(k)} \\
            0 & 1
        \end{array}\right),\; k \in \Omega_{12} \cup \Omega_{22},\\
        \left(\begin{array}{cc}
            1 & 0\\
            -\frac{r(k)}{1+|r(k)|^2} \delta_-^2(k) e^{2it \theta(k)} & 1
        \end{array}\right),\; k \in \Omega_{13} \cup \Omega_{23}.\\
    \end{cases}
\end{align}
The domains $\Omega_{11},\cdots,\Omega_{14}$ and $\Omega_{21},\cdots,\Omega_{24}$ can be seen in Figure \ref{rj1}. The function $M^{(2)}(k;x,t)$ solves the RHP \ref{m2} as follows:

\begin{figure}
	\begin{center}
		 \begin{tikzpicture}
		
\draw[Blue!10,fill=LightSteelBlue!] (-3.6,1.8)--(-1.8,0)--(-3.6,0);
\draw[CadetBlue!20,fill=LightSteelBlue!20] (-3.6,-1.8)--(-1.8,0)--(-3.6,0);	
\draw[CadetBlue!20,fill=LightSteelBlue!20] (-1.8,0)--(0,1.8)--(0,0);
\draw[Blue!10,fill=LightSteelBlue!] (-1.8,0)--(0,-1.8)--(0,0);
\draw[Blue!10,fill=LightSteelBlue!] (0,0)--(0,-1.8)--(1.8,0);
\draw[CadetBlue!20,fill=LightSteelBlue!20] (1.8,0)--(0,1.8)--(0,0);
\draw[Blue!10,fill=LightSteelBlue!] (3.6,1.8)--(1.8,0)--(3.6,0);
\draw[CadetBlue!20,fill=LightSteelBlue!20] (3.6,-1.8)--(1.8,0)--(3.6,0);	
		\draw[](-2.7,0.9)--(-1.8,0);
		\draw[->](-3.6,1.8)--(-2.7,0.9);

		\draw[](-2.7,-0.9)--(-1.8,0);
       \draw[->](-3.6,-1.8)--(-2.7,-0.9);

        \draw[->](-1.8,0)--(-0.9,0.9);
        \draw[](-0.9,0.9)--(0,1.8);

        \draw[->](-1.8,0)--(-0.9,-0.9);
        \draw[](-0.9,-0.9)--(0,-1.8);

		\draw[->](1.8,0)--(2.7,0.9);
		\draw[](3.6,1.8)--(2.7,0.9);
		
		\draw[->](1.8,0)--(2.7,-0.9);
		\draw[](3.6,-1.8)--(2.7,-0.9);
		
		\draw[](1.8,0)--(0.9,0.9);
		\draw[->](0,1.8)--(0.9,0.9);
		
		\draw[](1.8,0)--(0.9,-0.9);
		\draw[->](0,-1.8)--(0.9,-0.9);
		
				\node[scale=1] at (1.8,-0.3) {$k_0$};
			\node[scale=1] at (-1.8,-0.3) {$-k_0$};
		\node[scale=1] at (2.8,0.3) {$\Omega_{21}$};
	 	\node[scale=1] at (3.1,0.8) {$\Sigma_{21}$};
		\node[scale=1] at (2.8,-0.3) {$\Omega_{24}$};
		\node[scale=1] at (3.1,-0.8) {$\Sigma_{24}$};

		\node[scale=1] at (0.72,0.3) {$\Omega_{22}$};
		\node[scale=1] at (0.6,0.8) {$\Sigma_{22}$};
		\node[scale=1] at (0.72,-0.3) {$\Omega_{23}$};
		\node[scale=1] at (0.6,-0.8) {$\Sigma_{23}$};

		\node[scale=1] at (-0.72,0.3) {$\Omega_{12}$};
	   \node[scale=1] at (-0.6,0.8) {$\Sigma_{12}$};
		\node[scale=1] at (-0.72,-0.3) {$\Omega_{13}$};
		\node[scale=1] at (-0.6,-0.8) {$\Sigma_{13}$};
		\node[scale=1] at (-2.8,0.3) {$\Omega_{11}$};
		\node[scale=1] at (-3.1,0.8) {$\Sigma_{11}$};
		\node[scale=1] at (-2.8,-0.3) {$\Omega_{14}$};
	    \node[scale=1] at (-3.1,-0.8) {$\Sigma_{14}$};

		\draw[ ->](-4,0)--(4,0)node[black,right]{Re$z$};
		\draw[ ->](0,-2)--(0,2)node[black,right]{Im$z$};
		\end{tikzpicture}
    \end{center}
\caption{ \small The oriented contours and regions.   ${\rm Re} (2i\theta)>0$   in the blue regions and ${\rm Re} (2i\theta)<0$  in the gray regions.}
\label{rj1}
\end{figure}
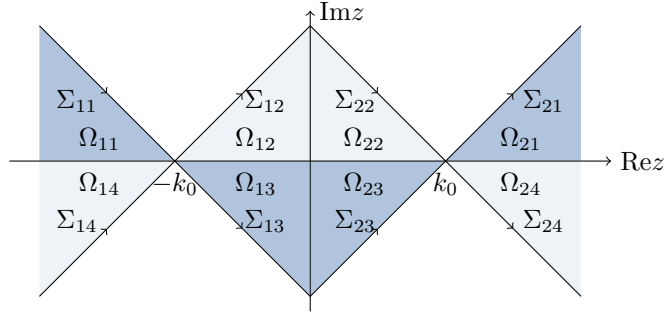

\begin{RHP}\label{m2}
	 Find a matrix-valued function $M^{(2)}(k;x,t)$ which satisfies
\begin{itemize}
    \item Analyticity: $M^{(2)}(k;x,t)$ is analytic in $\mathbb{C}\setminus \Sigma^{(2)}$, where $\Sigma^{(2)}= \cup_{i=1}^4 \left(\Sigma_{1i} \cup \Sigma_{2i} \right)$. See Figure \ref{rj1}.
    \item Jump condition: $M^{(2)}(k;x,t)$ has the jump condition
    \begin{equation}\label{V2}
        V^{(2)}(k;x,t)=\begin{cases}
            \left(\begin{array}{cc}
                1 & 0\\
                -r(k)  \delta^{-2}(k)e^{2it \theta(k)} & 1
            \end{array}\right),\; k \in \Sigma_{11} \cup \Sigma_{21},\\
           \left(\begin{array}{cc}
                1 & -\overline{r(\bar{k})} \delta^{2}(k) e^{-2it \theta(k)} \\
               0 & 1
            \end{array}\right),\; k\in \Sigma_{14} \cup \Sigma_{24},\\
            \left(\begin{array}{cc}
                1 & -\frac{\overline{r(\bar{k})}}{1+|r(k)|^2} \delta_+^2(k)^{-2it \theta(k)} \\
                0 & 1
            \end{array}\right),\; k \in \Sigma_{12} \cup \Sigma_{22},\\
            \left(\begin{array}{cc}
                1 & 0\\
                -\frac{r(k)}{1+|r(k)|^2} \delta_-^2(k) e^{2it \theta(k)} & 1
            \end{array}\right),\; k \in\Sigma_{13} \cup \Sigma_{23},\\
    \left(\begin{array}{cc}
        1 & 0\\
       f(k) \delta^{-2}(k)e^{2it \theta(k)} & 1
    \end{array}\right),\; k \in [E,-B/2],\\
     \left(\begin{array}{cc}
         1 & - \overline{f(\bar{k})} \delta^{2}(k)e^{-2it \theta(k)}\\
        0 & 1
     \end{array}\right),\; k \in [-B/2,\bar{E}].
 \end{cases}
 \end{equation}
    \item Asymptotic behavior: $M^{(2)}(k;x,t)\to I$, $k\to \infty$.
\end{itemize}
\end{RHP}
Since the jump matrix (\ref{V2}) decays to the identity matrix uniformly and exponentially fast, except for the neighbors of $k=k_0$ and $k=-k_0$, we match $M^{(2)}(k;x,t)$ with the local model $M^{lo}(k;x,t)$ in  the following way
\begin{equation*}
    M^{(2)}(k;x,t)=E(k;x,t)M^{lo}(k;x,t),
\end{equation*}
where $M^{lo}(k;x,t)$ is a solution of RHP \ref{mlo} and $E(k;x,t)$ is an error function.
Let $\epsilon>0$ and define the two small regions around the points $\pm k_0$ as
\begin{equation*}
     \mathcal{U}_{\pm k_0}:=\{  k \in \mathbb{C}| |k\pm k_0|\le \epsilon \},
\end{equation*}
and the corresponding boundaries of the regions $\mathcal{U}_{\pm k_0}$ as $\partial \mathcal{U}_{\pm k_0} $.
Denote  new contours  $\Sigma_{i}:= \cup_{j=1}^4 \Sigma_{ij} \cap \mathcal{U}_{\pm k_0}$, $i=1,2$ and $\Sigma^{lo}:=\Sigma^{(2)}\cap \mathcal{U}_{\pm k_0}=   \Sigma_{1}\cup  \Sigma_{2}$. See Figure \ref{rj2}.

\begin{figure}
	\begin{center}
		\begin{tikzpicture}
		
	    \draw[dashed](-1.8,0) circle (1.4142135623730950488016887242096980785696718753);
		\draw[](-2.3,0.5)--(-1.8,0);
		\draw[->](-2.8,1)--(-2.3,0.5);
		
		\draw[](-2.3,-0.5)--(-1.8,0);
		\draw[->](-2.8,-1)--(-2.3,-0.5);
		
		\draw[->](-1.8,0)--(-1.3,0.5);
		\draw[](-1.3,0.5)--(-0.8,1);
		
		\draw[->](-1.8,0)--(-1.3,-0.5);
		\draw[](-1.3,-0.5)--(-0.8,-1);

	  \draw[dashed](1.8,0) circle (1.4142135623730950488016887242096980785696718753);
		\draw[->](1.8,0)--(2.3,0.5);
		\draw[](2.8,1)--(2.3,0.5);
		
		\draw[->](1.8,0)--(2.3,-0.5);
		\draw[](2.8,-1)--(2.3,-0.5);
		
		\draw[](1.8,0)--(1.3,0.5);
		\draw[->](0.8,1)--(1.3,0.5);
		
		\draw[](1.8,0)--(1.3,-0.5);
		\draw[->](0.8,-1)--(1.3,-0.5);
		
			\coordinate (I) at (1.8,0);
\fill[black] (I) circle (1pt) node[below] {$k_0$};
					\coordinate (A) at (-1.8,0);
		\fill[black] (A) circle (1pt) node[below] {$-k_0$};
		\draw[ ->](-4,0)--(4,0)node[black,right]{Re$z$};
		\draw[ ->](0,-2)--(0,2)node[black,right]{Im$z$};
		\end{tikzpicture}
	\end{center}
	\caption{ \small The oriented contours of $M^{lo}$.}
	\label{rj2}
\end{figure}
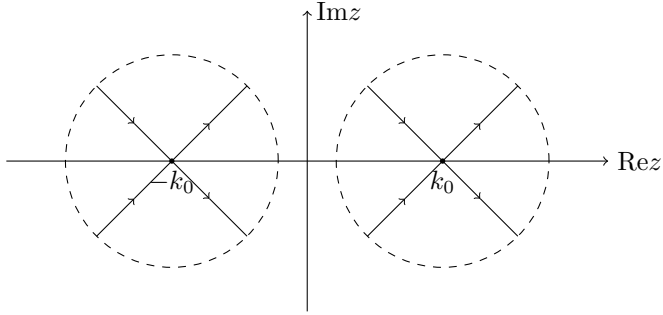

\begin{RHP}  \label{mlo}
	Find a matrix-valued function  $M^{lo}(k;x,t)$ such that
              \begin{itemize}
          	\item Analyticity: $M^{lo}(k;x,t)$ is analytic in $\mathbb{C}\setminus \Sigma^{lo}$.
          	\item Jump condition:
          	\begin{align*}
          		M^{lo}_-(k;x,t)=M^{lo}_+(k;x,t)V^{(2)}(k;x,t), \quad k\in \Sigma^{lo}.
          	\end{align*}
          	\item Asymptotic behavior:  	
          	\begin{align*}
          		M^{lo}(k;x,t)=I+\mathcal{O}(k^{-1}),	\quad  k \to  \infty.
          	\end{align*}
          \end{itemize}
\end{RHP}
Following the idea of constructing local models for the neighbors of the phase points in Section 5.1.2 of \cite{WF}, we  find the contribution to the solution of $M^{lo}(k;x,t)$ is  the sum of the separate contributions from $M^{lo,i}(k;x,t)$ in the contour $\Sigma_i$, $i = 1,2$.
It is obvious that  $M^{lo,i}(k;x,t)$  can be explicitly constructed by the parabolic cylinder functions. See more details in \cite{WF}.

\begin{RHP}
	 Find a matrix-valued function  $M^{lo,i} (k;x,t)$ such that
\begin{itemize}
   \item Analyticity: $M^{lo,i}(k;x,t)$ is analytic in $\mathbb{C}\setminus \Sigma_i$.
   \item Jump condition:
   \begin{align*}
       M^{lo,i}_+(k;x,t)=M^{lo,i}_-(k;x,t)V^{(2)}(k;x,t), \quad k\in \Sigma_{i}.
   \end{align*}
   \item Asymptotic behavior:  	
   \begin{align*}
       M^{lo,i}(k;x,t)=I+\mathcal{O}(k^{-1}),	\quad  k \to  \infty.
   \end{align*}
\end{itemize}
\end{RHP}
Note that $\theta$ is an odd function, which is $\theta(k)=-\theta(-k)$. For $k$ in the neighborhood of $\pm k_0$, we have
\begin{align*}
    \theta(k)&= 2(k+k_0)^3-6k_0 (k+k_0)^2 + 4 k_0^3,\\
    \theta(k)& =  2(k+k_0)^3+6k_0 (k+k_0)^2 - 4 k_0^3
\end{align*}
 We introduce the resealed variables
\begin{align*}
    s_1&=\sqrt{24 t k_0}(k+k_0),\\
    s_2&=\sqrt{24 t k_0}(k-k_0).
\end{align*}
Then we are in a situation where the asymptotic analysis of \cite{GM} works. Therefore, by the similar method in Section 6.1 of \cite{GM}, we obtain the asymptotic result in the Z-M region.

%In addition, following the step in the reference,  $E(k;x,t)$ can be estimated
%\begin{equation*}
%     E(k;x,t)=I+\mathcal{O}(t^{-1/2} \log t).
%\end{equation*}

\subsection{The plane wave region}
 For $\xi<-\frac{B^2}{4}$, $\im \theta(k)<0$ for $k\in(E,-\frac{B}{2})$ and $\im \theta(k)>0$ for $k\in(-\frac{B}{2},\bar{E})$, which imply the exponential term $e^{\pm i \theta}$ in the jump matrix $V(k;x,t)$ (\ref{V0}) increases with $t$ on the corresponding interval.
Then,  the jump matrix $V(k;x,t)$ does not converge as $t\to\infty$.
Hence, we have to use the modified nonlinear steepest descent method which involves replacing  the phase function $\theta(k)$ with a new phase function $g(k)$ and then transforming the original RHP into the model RHP of the finite-gap type. See more details in \cite{D1}-\cite{M141}.
We first consider the case  where the $g$-function transforms the original RHP to the model RHP of the zero-genus type.

Setting
\begin{equation}
      g(k;\xi):= \Omega(k)+6 \xi X(k)=\left( 2k^2 -Bk+\frac{B^2}{2}-A^2+6\xi\right) X(k),
\end{equation}
where $X(k)$ and $\Omega(k)$ have been defined in (\ref{XO}).
The differential $\mathrm{d} g(k;\xi)$ can be written in the form
\begin{align*}
    \mathrm{d} g(k;\xi)&= \frac{6k^3+3Bk^2+\left(3A^2+6\xi\right)k-3A^2B/2+3\xi B}{X(k)} \mathrm{d} k \\
    &= 6\frac{\left(k-P_1(\xi)\right)\left(k-P_2(\xi)\right)\left(k-P_3(\xi)\right)}{X(k)}\mathrm{d} k.
\end{align*}
Analysing  the property of  the  denominator of $ \mathrm{d} g(k;\xi)$, we find it has three different real roots if
\begin{align}
   \Lambda (\xi):= \frac{\xi^3}{27}+\left(\frac{A^2}{18}+\frac{B^2}{54}\right)\xi^2 +\left( \frac{A^4}{36}-\frac{7A^2B^2}{108}+\frac{B^4}{432}\right)\xi +\frac{A^6}{216}-\frac{A^2 B^4}{864}+\frac{11A^4B^2}{432}<0,
\end{align}
which is true in the region $\xi<-\frac{B^2}{4}$.
%In addition, $ \Lambda (-\infty)<0$ and  $ \Lambda (+\infty)>0$. Then, the interval of $\xi$ for $\Lambda(\xi)<0$ and $\xi<-\frac{B^2}{4}$ exists.
Therefore,  the formula of $P_i(\xi)$ for $i=1,2,3$ can be expressed as
\begin{align*}
    P_1(\xi)= 2 \sqrt[3]{r} \cos \theta_0, \; P_2(\xi)= 2 \sqrt[3]{r} \cos \left( \theta_0+\frac{2\pi}{3}\right),\; P_3(\xi)= 2 \sqrt[3]{r} \cos \left( \theta_0+\frac{4\pi}{3}\right),
\end{align*}
where
\begin{align*}
    r=\sqrt{-\left( \frac{A^2}{6}+\frac{\xi}{3}-\frac{B^2}{36} \right)^3},\; \theta_0=\frac{1}{3}\arccos \left( - \frac{36A^2B+36B\xi+B^3}{216 r}\right).
\end{align*}
$g(k;\xi)$ has the following properties:
\begin{itemize}
    \item $g(k;\xi)$ is analytic in $\mathbb{C}\setminus \left[E,\bar{E}\right]$.
    \item $g(k;\xi)=\theta(k)+g(\infty;\xi)+\mathcal{O}(k^{-1}),\; k\to \infty$,
    where $g(\infty;\xi) = \frac{-6A^2 B + B^3 + 12 B \xi}{4} $.

\begin{figure}
	\begin{center}

	\begin{tikzpicture}
\node[anchor=south west,inner sep=0](image) at (7,0)
{\centering\includegraphics[width=0.6\textwidth]{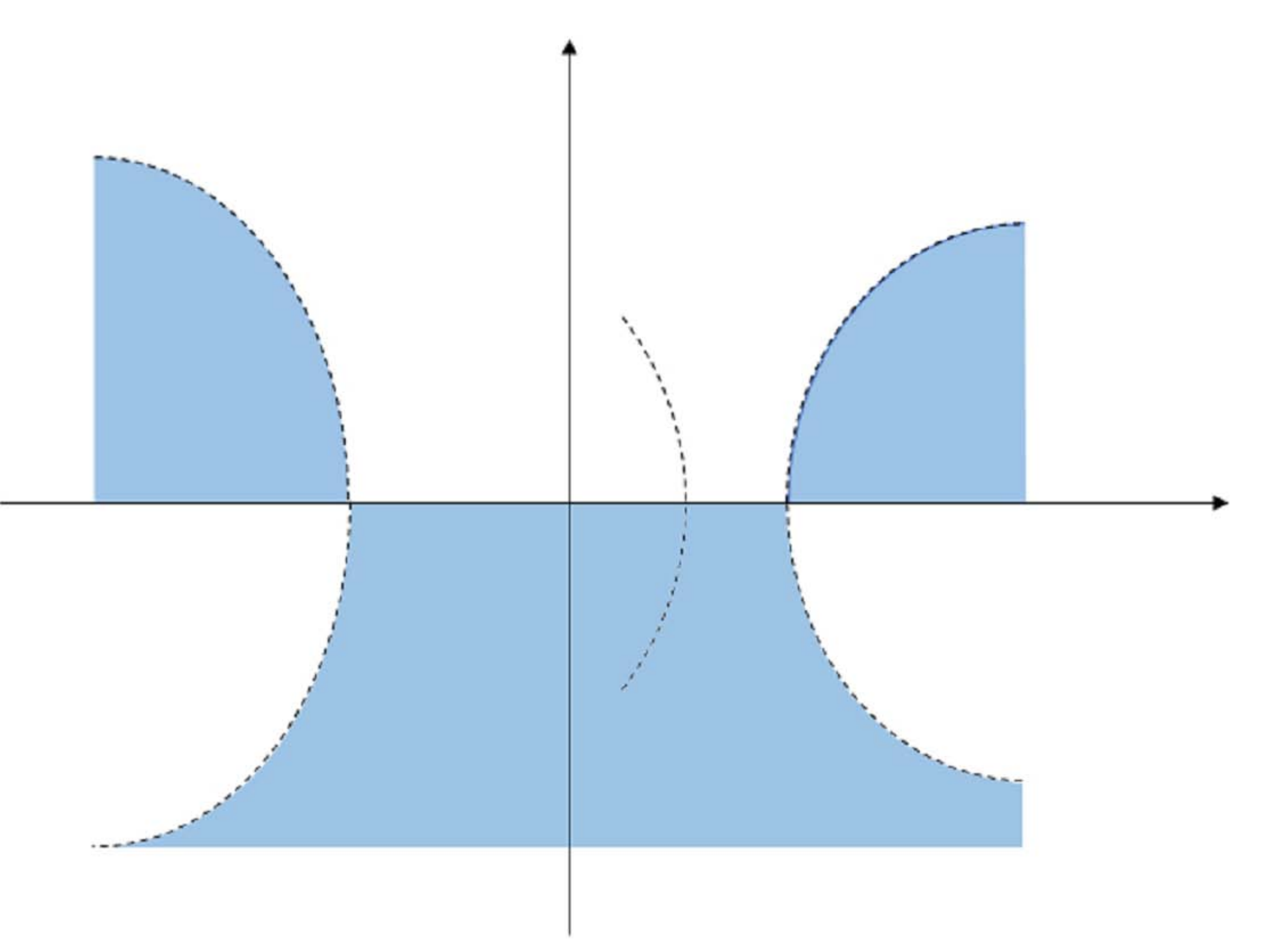}};
\begin{scope}[x={(image.south east)},y={(image.north west)}]

%	\node at (0.25,0.7) {\footnotesize$|e^{2 i t \theta}| \rightarrow 0$};
%	\node at (0.52,0.7) {\footnotesize$|e^{2 i t \theta}| \rightarrow 0$};
%	\node at (0.75,0.28) {\footnotesize$|e^{-2 i t \theta}| \rightarrow 0$};
%    \node at (0.33,0.82) {\footnotesize$|e^{-2 i t \theta}| \rightarrow 0$};
%	\node at (0.49,0.26) {\footnotesize$|e^{-2 i t \theta}| \rightarrow 0$};
%	\node at (0.68,0.15) {\footnotesize$|e^{2 i t \theta}| \rightarrow 0$};

\node at (0.24,0.65) {$+$};
\node at (0.55,0.65) {$+$};
\node at (0.72,0.28) {$-$};
\node at (0.325,0.82) {$-$};
\node at (0.42,0.26) {$-$};
\node at (0.63,0.15) {$+$};	
\node at (0.4,0.42) {$P_1$};	
\node at (0.56,0.42) {$P_2$};	
\node at (0.61,0.42) {$P_3$};	
\node at (0.46,0.57) {$\gamma_g$};
\node at (0.556,0.36) {$\bar{\gamma}_g$};
\end{scope}		
\end{tikzpicture}
    \end{center}
\caption{ \small The sign and boundary of $\im g(k)=0$. The signal $+$ stands for $\im g(k)>0$ in the blue region, while the signal $-$ stands for $\im g(k)<0$ in the white region.}
\label{rj5}
\end{figure}
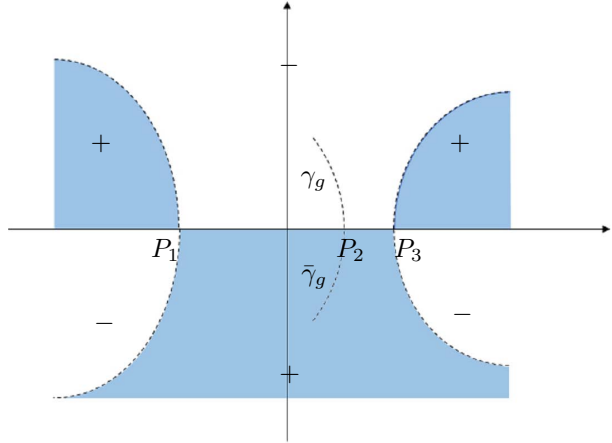

    \item $g(k;\xi)=\mathcal{O}\left(k+B/2 \mp iA  \right),\; k \to -B/2\pm iA$.
     \item  $g_+(k;\xi)=-g_-(k;\xi),\; k \in \gamma_g \cup \bar{\gamma}_g$, where $\gamma_g \cup \bar{\gamma}_g$ is a finite arc connecting the branch points $E$ and $\bar{E}$ such that $\im g(k;\xi)=0$ for all $k$ along this arc.  More importantly, we can obtain the boundary of $\im g(k;\xi)=0$ through controlling the region which satisfies $\im g^2(k;\xi)=0$ and $\re g^2(k;\xi)>0$. See Figure \ref{rj5}.

\end{itemize}

In the plane wave region $\{\xi \in \mathbb{R}^-|\Lambda (\xi)<0, \xi<-\frac{B^2}{4}  \}$,
     we solve the RHP for the matrix $M(k;x,t)$  by 4 transformations.
     Let us define a new matrix,
     \begin{equation*}
        M^{(1)}(k;x,t)= e^{-itg(\infty;\xi)\sigma_3}  M(k;x,t) e^{it(g(k)-\theta(k))\sigma_3}.
     \end{equation*}
     Then the matrix-value function  $M^{(1)}(k;x,t)$ solves the following RHP:

    \begin{RHP}  Find a matrix-valued function $M^{(1)}(k;x,t)$ which satisfies
     \begin{itemize}
         \item Analyticity: $M^{(1)}(k;x,t)$ is analytic in $\mathbb{C}\setminus \Sigma^{(1)}$, where $\Sigma^{(1)}=\mathbb{R}\cup \gamma_g \cup \bar{\gamma}_g$.
         \item Jump condition: $M^{(1)}(k;x,t)$ has the jump condition $$M^{(1)}_-(k;x,t)=M^{(1)}_+(k;x,t)V^{(1)}(k;x,t), \; k \in \Sigma^{(1)},$$
         where
         \begin{equation}\label{2V1}
            V^{(1)}(k;x,t)=\begin{cases}
             \left(\begin{array}{cc}
                 1 & -\overline{r(k)}     e^{-2it g(k)}\\
                 -r(k)e^{2it g(k)} & 1+|r(k)|^2
             \end{array}\right),\; k \in \mathbb{R},\\
         \left(\begin{array}{cc}
            e^{-2it g(k)} & 0\\
            f(k) & e^{2it g(k)}
         \end{array}\right),\; k \in \gamma_g,\\
         \left(\begin{array}{cc}
            e^{-2it g(k)} & - \overline{f(\bar{k})}\\
            0 & e^{2it g(k)}
         \end{array}\right),\; k \in \bar{\gamma}_g,
     \end{cases}
     \end{equation}
     with $f(k)$ being given in (\ref{f}).
         \item Asymptotic behavior:
         \begin{align*}
                 M^{(1)}(k;x,t)=I+\mathcal{O}(k^{-1}),	\quad  k \to  \infty.
         \end{align*}
     \end{itemize}
 \end{RHP}
 For the same reason in the first transformation applied in the Z-M region, we introduce the second transformation here
\begin{align*}
     M^{(2)}(k;x,t)=M^{(1)}(k;x,t) \delta^{-\sigma_3}(k),
\end{align*}
where
\begin{align*}
    \delta(k)&=\exp \left[ \frac{1}{2\pi i} \int_{P_1(\xi)}^{P_3(\xi)} \frac{ \log (1 +|r(s)|^2)}{s-k} \mathrm{d} s\right]\\
    & =\left(\frac{k-P_3(\xi)}{k-P_1(\xi)}\right)^{-iv(P_3(\xi))} \exp \left[ \frac{1}{2\pi i} \int_{P_1(\xi)}^{P_3(\xi)} \frac{ \log \left( \frac{1+|r(s)|^2}{1+|r(P_3(\xi))|^2}\right)}{s-k} \mathrm{d} s\right]
\end{align*}
and \begin{align*}
    v(k)=\frac{1}{2\pi} \log \left( 1+|r(k)|^2 \right).
\end{align*}
In addition, $\delta_-=\delta_+\left( 1+|r|^2 \right)^{-1}$ for $k \in \left[P_1(\xi), P_3(\xi)\right]$.
$M^{(2)}(k;x,t)$ satisfies the following RHP:

\begin{RHP}
  Find a matrix-valued function $M^{(2)}(k;x,t)$ which satisfies
\begin{itemize}
    \item Analyticity: $M^{(2)}(k;x,t)$ is analytic in $\mathbb{C}\setminus \Sigma^{(2)}$, where $\Sigma^{(2)}=\Sigma^{(1)}$.
    \item Jump condition: $M^{(2)}_-(k;x,t)=M^{(2)}_+(k;x,t)V^{(2)}(k;x,t), \; k \in \Sigma^{(2)}$,
    where
    \begin{equation}
       V^{(2)}(k;x,t)= \begin{cases}\left(\begin{array}{cc}
            1 & -\overline{r(k)} \delta^2(k)    e^{-2it g(k)}\\
            -r(k) \delta^{-2}(k) e^{2it g(k)} & 1+|r(k)|^2
        \end{array}\right),\; k \in \mathbb{R} \setminus \left( P_1(\xi),P_3(\xi)\right),\\
      \left(\begin{array}{cc}
            1+|r(k)|^2 & -\frac{\overline{r(k)}}{1+|r(k)|^2  }  \delta^2_+(k)    e^{-2it g(k)}\\
            -\frac{r(k)}{1+|r(k)|^2} \delta^{-2}_-(k) e^{2it g(k)} & 1
        \end{array}\right),\; k \in \left( P_1(\xi),P_3(\xi)\right),\\
    \left(\begin{array}{cc}
       e^{-2it g(k)} & 0\\
       f(k) & e^{2it g(k)}
    \end{array}\right),\; k \in \gamma_g,\\
    \left(\begin{array}{cc}
       e^{-2it g(k)} & - \overline{f(\bar{k})}\\
       0 & e^{2it g(k)}
    \end{array}\right),\; k \in \bar{\gamma}_g.
\end{cases}
\end{equation}
\item Asymptotic behavior:
    \begin{align*}
            M^{(2)}(k;x,t)=I+\mathcal{O}(k^{-1}),	\quad  k \to  \infty.
    \end{align*}
\end{itemize}
\end{RHP}
Then the subsequent transformation is
\begin{align*}
    M^{(3)}(k;x,t)=M^{(2)}(k;x,t)G^{(2)}(k;x,t),
\end{align*}
where
\begin{align*}
    G^{(2)}(k;x,t)=\begin{cases}
        \left(\begin{array}{cc}
            1 &0\\
            -r(k) \delta^{-2}(k) e^{2it g(k)} & 1
        \end{array}\right),\; k \in \Omega_1 \cup \Omega_2,\\
        \left(\begin{array}{cc}
            1 & \overline{r(\bar{k})} \delta^2(k)    e^{-2it g(k)}\\
           0 & 1
        \end{array}\right),\; k \in \Omega_3 \cup \Omega_4,\\
        \left(\begin{array}{cc}
            1 & -\frac{\overline{r(\bar{k})}}{1+|r(k)|^2  }  \delta^2_+(k)    e^{-2it g(k)}\\
           0 & 1
        \end{array}\right),\; k \in \Omega_5,\\
        \left(\begin{array}{cc}
            1 & 0\\
            \frac{r(k)}{1+|r(k)|^2} \delta^{-2}_-(k) e^{2it g(k)} & 1
        \end{array}\right),\; k \in \Omega_6,\\
        \left(\begin{array}{cc}
            1 &0\\
            0& 1
        \end{array}\right),\; k \in \Omega_7 \cup \Omega_8.
    \end{cases}
\end{align*}
This transformation gives the following RHP:

\begin{figure}
	\begin{center}
    \begin{tikzpicture}
    \draw [dashed](-4,0)--(4,0);
    \draw [dashed](0,-2.2)--(0,2.2);
     \draw [->](-2,0)--(-1,0);
     \draw [->](-1,0)--(0.6,0);
     \draw [](0.6,0)--(1.5,0);
%    \draw [](0.5,1) parabola (1,0);
%    \draw [](0.5,-1) parabola (1,0);

    \draw [](0.8,0) arc(0:30:2);
    \draw [->](0.532,1) arc(30:15:2);
     \draw [->](0.8,0) arc(360:345:2);
    \draw [](0.8,0) arc(360:330:2);

%    \draw [](3.6,2) parabola (1.5,0);
%     \draw [](3.6,-2) parabola (1.5,0);
      \draw [](3.6,2)--(2.55,1);
      \draw [->](1.5,0) --(2.55,1);
      \draw [](3.6,-2)--(2.55,-1);
      \draw [->](1.5,0) --(2.55,-1);
%     \draw [](-3.6,2) parabola (-2,0);
%     \draw [](-3.6,-2) parabola (-2,0);

      \draw [->](-3.6,2)--(-2.8,1);
      \draw [](-2.8,1) --(-2,0);
      \draw [->](-3.6,-2)--(-2.8,-1);
       \draw [](-2.8,-1) --(-2,0);
       \draw[->] (-2,0) arc (180:270:1.75);
        \draw[] (1.5,0) arc (0:360:1.75);
       \draw[->] (-2,0) arc (180:90:1.75);
%     \draw[fill=LightSteelBlue!] (-3.6,2) parabola (-2,0);
%    \draw[LightSteelBlue!,fill=LightSteelBlue!](-3.6,2)--(-2,0)--(-3.6,0);

%    \draw[fill=LightSteelBlue!] (3.6,2) parabola (1.5,0);
%    \draw[LightSteelBlue!,fill=LightSteelBlue!](3.6,2)--(1.5,0)--(3.6,0);

%  \draw[fill=LightSteelBlue!] (3.6,-2) parabola (1.5,0);
% \draw[LightSteelBlue!,fill=LightSteelBlue!](3.6,-2)--(1.5,0)--(3.6,0);

%       \draw []  (1.5,0) parabola[bend at end] (3.6,-2);

	\node[scale=1] at (3,0.6) {$\Omega_1$};
	\node[scale=1] at (-3,0.6) {$\Omega_2$};
	\node[scale=1] at (-0.2,-0.6) {$\Omega_6$};	
	\node[scale=1] at (3,-0.6) {$\Omega_4$};
	\node[scale=1] at (-0.2,0.6) {$\Omega_5$};
	\node[scale=1] at (-3,-0.6) {$\Omega_3$};

\node[scale=1] at (3,1.8) {$L_1$};
\node[scale=1] at (-3,1.8) {$L_2$};
\node[scale=1] at (-0.2,-2) {$L_6$};	
\node[scale=1] at (3,-1.8) {$L_4$};
\node[scale=1] at (-0.2,1.92) {$L_5$};
\node[scale=1] at (-3,-1.8) {$L_3$};  	
\node[scale=1] at (0.9,0.9) {$\gamma_g$};
\node[scale=1] at (0.9,-0.9) {$\bar{\gamma}_g$}; 	
%	   			\coordinate (c) at (-2,0);
%	   \fill[black] (c) circle (1pt) node[below] {$P_1$};
%	   	   			\coordinate (d) at (0.8,0);
%	   \fill[black] (d) circle (1pt) node[below] {$P_2$};
%	   	   			\coordinate (e) at (1.5,0);
%\fill[black] (e) circle (1pt) node[below] {$P_3$};	

	   	\node[scale=1] at (-2.4,-0.2) {$P_1$};
	 	\node[scale=1] at (1,-0.2) {$P_2$};
	    \node[scale=1] at (1.9,-0.2) {$P_3$};
    \end{tikzpicture}
    \end{center}
\caption{ \small The contour $ \Sigma^{(3)}$ for $M^{(3)}(k;x,t)$.}
\label{rj4}
\end{figure}
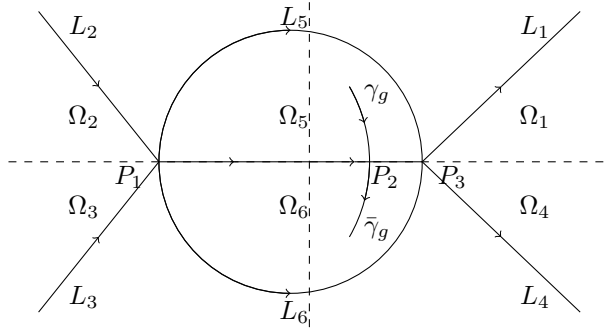

\begin{RHP}  Find a matrix-valued function $M^{(3)}(k;x,t)$ which satisfies
\begin{itemize}
    \item Analyticity: $M^{(3)}(k;x,t)$ is analytic in $\mathbb{C}\setminus \Sigma^{(3)}$, where $\Sigma^{(3)}= \cup_{i=1}^6 L_i$. See Figure \ref{rj4}.
    \item Jump condition: $M^{(3)}_-(k;x,t)=M^{(3)}_+(k;x,t)V^{(3)}(k;x,t), \; k \in \Sigma^{(3)}$,
    where
    \begin{equation}
       V^{(2)}(k;x,t)= \begin{cases}\left(\begin{array}{cc}
            1 &0\\
            -r(k) \delta^{-2}(k) e^{2it g(k)} & 1
        \end{array}\right),\; k \in L_1 \cup L_2,\\
        \left(\begin{array}{cc}
            1 & \overline{r(\bar{k})} \delta^2(k)    e^{-2it g(k)}\\
           0 & 1
        \end{array}\right),\; k \in L_3 \cup L_4,\\
       \left(\begin{array}{cc}
            1 & -\frac{\overline{r(\bar{k})}}{1+|r(k)|^2  }  \delta^2_+(k)    e^{-2it g(k)}\\
           0 & 1
        \end{array}\right),\; k \in L_5,\\
        \left(\begin{array}{cc}
           1 & 0\\
           -\frac{r(k)}{1+|r(k)|^2}  \delta^{-2}_-(k) e^{2it g(k)}& 1
        \end{array}\right),\; k \in L_6,\\
    \left(\begin{array}{cc}
       0 & -f^{-1}(k) \delta^2(k)\\
       f(k) \delta^{-2}(k) & 0
    \end{array}\right),\; k \in \gamma_g,\\
    \left(\begin{array}{cc}
      0 & - \overline{f(\bar{k})} \delta^2(k)\\
       \left( \overline{f(\bar{k})} \right)^{-1}\delta^{-2}(k) & 0
    \end{array}\right),\; k \in \bar{\gamma}_g,
\end{cases}
\end{equation}
after using the relations
\begin{align*}
    1-\overline{a_-(\bar{k})b_-(\bar{k})} f(k)=0,\; 1+\overline{f(\bar{k})} a_+(k)b_+(k)=0.
\end{align*}
\item Asymptotic behavior:
    \begin{align*}
            M^{(3)}(k;x,t)=I+\mathcal{O}(k^{-1}),	\quad  k \to  \infty.
    \end{align*}
\end{itemize}
\end{RHP}
 We notice that $V^{(3)}(k;x,t)=I+\mathcal{O}\left(e^{-\epsilon t}\right)$ as $t \to \infty$ for $k \in L_i(i=1,2,\cdots,6)$ with the exception of small neighborhoods of the stationary points $P_1(\xi)$ and $P_3(\xi)$.
 Hence,  the jump matrix for $k \in \gamma_g \cup \bar{\gamma}_g$ is the main contribution of the asymptotics and it can be factorized as
 \begin{align}\label{2V3}
    V^{(3)}(k;x,t)=  \left(\begin{array}{cc}
        F_+^{-1}(k) & 0\\
        0 & F_+(k)
     \end{array}\right) \left(\begin{array}{cc}
        0 & i\\
        i & 0
     \end{array}\right)\left(\begin{array}{cc}
        F_-(k) & 0\\
        0 & F_-^{-1}(k)
     \end{array}\right),
 \end{align}
 which takes place if
 \begin{equation*}
    F_-(k)F_+(k)= \begin{cases}
       -i f(k) \delta^{-2}(k),\; k \in \gamma_g,\\
       i f^{-1}(k) \delta^{-2}(k),\; k \in \bar{\gamma}_g.
    \end{cases}
\end{equation*}

Therefore, we arrive at the following scalar RHP:

\begin{RHP}  Find a scalar function $F(k)$ such that
 \begin{itemize}
     \item Analyticity: $F(k)$ is analytic in $\mathbb{C}\setminus \{ \gamma_g \cup \bar{\gamma}_g \}$.
    \item Boundedness: $F(k)$ is bounded at infinity.
     \item Jump condition: \begin{equation*}
         F_-(k)F_+(k)= \begin{cases}
            -i f(k) \delta^{-2}(k) = a^{-1}_-(k)a^{-1}_+(k) \delta^{-2}(k),\; k \in \gamma_g,\\
            i f^{-1}(k) \delta^{-2}(k) = \overline{a_+(\bar{k})a_-(\bar{k})}\delta^{-2}(k),\; k \in \bar{\gamma}_g.
         \end{cases}
     \end{equation*}
\end{itemize}
\end{RHP}
Let
\begin{equation*}
    H(k)=\begin{cases}
        a(k)F(k),\; k \in \mathbb{C}^+ \backslash \gamma_g,\\
        \frac{F(k)}{a(k)},\; k \in \mathbb{C}^- \backslash \bar{\gamma}_g.
    \end{cases}
\end{equation*}
Then we obtain
\begin{equation*}
    \left[ \frac{\log H(k)}{X(k)}\right]_+-\left[ \frac{\log H(k)}{X(k)}\right]_- = \begin{cases}
       \frac{ \log \delta^{-2}(k)}{X_+(k)},\; k\in \gamma_g \cup \bar{\gamma}_g,\\
       \frac{\log a^2(k)}{X(k)},\; k \in \mathbb{R}.
    \end{cases}
\end{equation*}
Using the Sokhotski-Plemelj formula, we have
\begin{equation}
    H(k)=\exp \left\{ \frac{X(k)}{2 \pi i}  \left[  \int_{\gamma_g \cup \bar{\gamma}_g} \frac{\log \delta^{-2}(s)}{s-k} \frac{\mathrm{d} s}{X_+(s)}  + \int_{\mathbb{R}} \frac{\log a^2(s)}{s-k} \frac{\mathrm{d}s}{X(s)} \right] \right\}.
\end{equation}
Since $a(\infty)=1$, $F(\infty)=H(\infty)=e^{i\phi(\xi)}$ with
\begin{align}\label{phi}
    \phi(\xi)=\frac{1}{2\pi}  \left[ \int_{\gamma_g \cup \bar{\gamma}_g} \frac{\log \delta^{-2}(s)}{X_+(s)} \mathrm{d} s  + \int_{\mathbb{R}} \frac{\log |a(s)|^2}{X(s)} \mathrm{d} s \right].
\end{align}
The factorization (\ref{2V3}) suggests the next step
\begin{align*}
    M^{(4)}(k;x,t)=F^{\sigma_3}(\infty)M^{(3)}(k;x,t) F^{-\sigma_3}(k),
\end{align*}
which satisfies the RHP as follows:

\begin{RHP} Find a matrix function $M^{(4)}(k;x,t)$ such that
\begin{itemize}
    \item Analyticity: $M^{(4)}(k;x,t)$ is analytic in $\mathbb{C}\backslash \Sigma^{(4)}$, where $\Sigma^{(4)}=\cup_{i=1}^4 L_i \cup \gamma_g \cup \bar{\gamma}_g$.
    \item Jump condition: $M^{(4)}_-(k;x,t)=M^{(4)}_+(k;x,t)V^{(4)}(k;x,t), \; k \in \Sigma^{(4)}$,
    where
    \begin{equation}
       V^{(4)}(k;x,t)= \begin{cases}\left(\begin{array}{cc}
            0&i\\
           i & 0
        \end{array}\right),\; k \in \gamma_g \cup \bar{\gamma}_g,\\
     I +\mathcal{O}\left(e^{-\epsilon t} \right),\; k \in \cup_{i=1}^4 L_i.
\end{cases}
\end{equation}
\item Asymptotic behavior:
    \begin{align*}
            M^{(4)}(k;x,t)=I+\mathcal{O}(k^{-1}),	\quad  k \to  \infty.
    \end{align*}
\end{itemize}
\end{RHP}
Finally, we match $M^{(4)}(k;x,t)$ to $M^{(mod)}(k;x,t)$ by
\begin{align*}
    M^{(4)}(k;x,t)=E(k;x,t)M^{(mod)}(k;x,t)
\end{align*}
where $E(k;x,t)$ is the error function and $M^{(mod)}(k;x,t)$ solves the model RHP:

\begin{RHP}\label{RHP311}  Find a matrix function $M^{(mod)}(k;x,t)$ such that
\begin{itemize}
    \item Analyticity: $M^{(mod)}(k;x,t)$ is analytic in $\mathbb{C}\backslash \{\gamma_g \cup \bar{\gamma}_g\}$.
    \item Jump condition: $M^{(mod)}_-(k;x,t)=M^{(mod)}_+(k;x,t) \left(\begin{array}{cc}
            0&i\\
           i & 0
        \end{array}\right),\; k \in \gamma_g \cup \bar{\gamma}_g.$
\item Asymptotic behavior: $M^{(mod)}(k;x,t)=I+\mathcal{O}(k^{-1}),	\quad  k \to  \infty.$
\end{itemize}
\end{RHP}
The solution of RHP \ref{RHP311} can be given explicitly by
\begin{equation*}
    M^{(mod)}(k;x,t)= \frac{1}{2 }\left(\begin{array}{cc}
        \varphi(k)+\frac{1}{ \varphi(k)} &  \varphi(k)-\frac{1}{ \varphi(k)}\\
        \varphi(k)-\frac{1}{ \varphi(k)} & \varphi(k)+\frac{1}{ \varphi(k)}
        \end{array}\right),
\end{equation*}
where $\varphi(k)$ has been defined in (\ref{var}).
In addition, we can find $E(k;x,t)=I+\mathcal{O}(k^{-1/2})$.
Therefore,  taking into account the inverse transformations of the four transformations above, we have
\begin{align*}
    u(x,t)=2i \lim_{k\to\infty} \left[ k M(k;x,t)\right]_{12}= 2i e^{2ig(\infty)t}m^{(mod)}_{12}(x,t)F^{-2}(\infty) +\mathcal{O}(t^{-1/2}),
\end{align*}
where $g(\infty)=\frac{B^3}{4}-\frac{3A^2B}{2}+3\xi B$ and $m^{(mod)}_{12}(x,t)==\lim_{k \to \infty}\left[ k M(k;x,t)\right]_{12}=-\frac{iA}{2}$. After simplification we can obtain the result in this region.

\subsection{The slow decay region}
In this section,  we consider the region $\xi > \frac{A^2}{3}-\frac{B^2}{4}$ under $\xi >0$.
%which is obtained from the following calculation:
%Let $\widetilde{A}$ be a real constant such that $\im \tilde{\theta}\left(-\frac{B}{2}+i \widetilde{A}\right)=0$, %which converts to
%\begin{align}
%    -2 \widetilde{A}^3 + 6 \left(\xi + \frac{B^2}{4}\right) \widetilde{A} = 0.
%\end{align}
%The above equation has three real constants
%\begin{align*}
%     \widetilde{A}_1 = \sqrt{3\left( \xi + \frac{B^2}{4} \right)},\;
%     \widetilde{A}_2 = -\sqrt{3 \left( \xi + \frac{B^2}{4} \right)},\;
%     \widetilde{A}_3 = 0.
%\end{align*}
%Here we consider the case that $\widetilde{A} = \widetilde{A}_1$ with $A<\widetilde{A}$, which is $ \xi > %\frac{A^2}{3}-\frac{B^2}{4}$.
In this case,  the interval $[ E,-B/2 ]$ is in the region $\im \theta(k)>0$ and $[-B/2,\bar{E} ]$ is in the region $\im \theta(k)<0$.

    Considering the decay in different areas, we divide the whole area into $3$ different areas, which are $\frac{A^2}{3} - \frac{B^2}{4} < \xi < \frac{A^2}{6}$, $ \frac{A^2}{6}<\xi<A^2$ and $\xi > A^2$, and estimate the $u(x,t)$ in these regions respectively.
    We  give the proof for  $\frac{A^2}{3} - \frac{B^2}{4} < \xi < \frac{A^2}{6}$ and other cases can be proved similarly.

    If $\frac{A^2}{3} - \frac{B^2}{4} < \xi < \frac{A^2}{6}$, we define the contours $L_1 = \{ k \in \mathbb{C}: k = k_1 + i \widehat{A}(k_1),\;-\infty < k_1 < +\infty,\; \text{$\widehat{A}$ varies with $k_1$} \}$  which satisfies
   \begin{itemize}
       \item $L_1$ is in the region $\im \theta(k)>0$;
       \item $L_1$  is not crossing the interval $[E,-B/2]$ and the value $L_1|_{-B/2} > A$,
   \end{itemize}
  and $L_2= \overline{L_1}$.
   Then we define the domains $\Omega_1 = \{k\in \mathbb{C}: k = \hat{k} + i \widehat{A}\; \text{with} \; \hat{k}<k_1\; \text{and}\; \widehat{A}<\widehat{A}(k_1)  \}$, $\Omega_2 = \overline{\Omega_1}$, $\Omega_3 = \{ k\in \mathbb{C}: k = \hat{k}+ i \widehat{A}\; \text{with}\; \hat{k}>k_1\; \text{and} \; \widehat{A}>\widehat{A}(k_1)  \}$ and $\Omega_4 = \overline{\Omega_3}$. See Figure \ref{rj6}.

  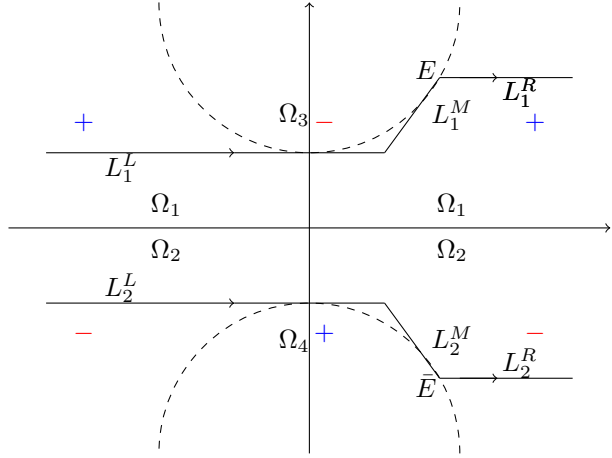
\begin{figure}
  	\begin{center}
  		\begin{tikzpicture}
  		  \draw [->](-4,0)--(4,0);
  		\draw [->](0,-3)--(0,3);
  		       \draw[dashed] (-2,3) arc (180:360:2);
  		       \draw[dashed] (2,-3) arc (0:180:2);
  		       \draw[->](-3.5,1)--(-1,1);
  		       \draw[](-1,1)--(1,1);
  		       \draw[](1,1)--(1.732,2);
  		       \draw[->](1.732,2)--(2.5,2);
  		       \draw[](2,2)--(3.5,2);

  		         \node[scale=1] at (2.8,1.8) {$L_1^R$};

  		       \node[scale=1] at (2.8,1.8) {$L_1^R$};
  		       \node[scale=1] at (2.8,-1.8) {$L_2^R$};
  		       \node[scale=1] at (-2.5,0.8) {$L_1^L$};
  		       \node[scale=1] at (-2.5,-0.8) {$L_2^L$};
  		        \node[scale=1] at (1.9,1.5) {$L_1^M$};
  		       \node[scale=1] at (1.9,-1.5) {$L_2^M$};

  		        \node[scale=1] at (1.9,0.3) {$\Omega_1$};
  		   \node[scale=1] at (-1.9,0.3) {$\Omega_1$};
  		    \node[scale=1] at (1.9,-0.3) {$\Omega_2$};
  		  \node[scale=1] at (-1.9,-0.3) {$\Omega_2$};
  		   		
  		 \node[scale=1] at (-0.2,1.5) {$\Omega_3$};
  		 \node[scale=1] at (-0.2,-1.5) {$\Omega_4$};
  		
  		   		 \node[scale=1] at (1.55,2.1) {$E$};
  		  		   		 \node[scale=1] at (1.55,-2.1) {$\bar{E}$};
  		\draw[->](-3.5,-1)--(-1,-1);
  		\draw[](-1,-1)--(1,-1);
  		\draw[](1,-1)--(1.732,-2);
  		\draw[->](1.732,-2)--(2.5,-2);
  		\draw[](2,-2)--(3.5,-2);
  		
        \node[blue,scale=1] at (3,1.4) {$+$};
  		\node[red,scale=1] at (0.2,1.4) {$-$};
  		        \node[blue,scale=1] at (-3,1.4) {$+$};
  		
  		        \node[red,scale=1] at (3,-1.4) {$-$};
  		\node[blue,scale=1] at (0.2,-1.4) {$+$};
  		  		        \node[red,scale=1] at (-3,-1.4) {$-$};
  		
  	    \end{tikzpicture}
  \end{center}
\caption{ \small The contour of  $M^{(1)}(k;x,t)$ for  $\frac{A^2}{3} - \frac{B^2}{4} < \xi < \frac{A^2}{6}$, where the dotted line is the boundary of $\im \theta =0$. The signal $+$ stands for $\im \theta(k)>0$ in the corresponding region, while the signal $-$ stands for $\im \theta(k)<0$ in the corresponding region. }
\label{rj6}
\end{figure}

   Further we make the following transformations to solve the initial problem (\ref{cmkdv})-(\ref{sl}):
   \begin{equation*}
    M^{(1)}(k;x,t)= M(k;x,t) G^{(1)}(k;x,t)
 \end{equation*}
 where
 \begin{align*}
      G^{(1)}(k;x,t) = \begin{cases}
        \left(\begin{array}{cc}
            1 & 0 \\
            -r(k)e^{2it \theta(k)} & 1
        \end{array}\right),\; k \in \Omega_1, \\
        \left(\begin{array}{cc}
            1 & \overline{r(k)}     e^{-2it \theta(k)}\\
           0 & 1
        \end{array}\right),\; k \in \Omega_2, \\
        \left(\begin{array}{cc}
            1 & 0\\
           0 & 1
        \end{array}\right),\; k \in \Omega_3 \cup \Omega_4.
      \end{cases}
 \end{align*}
$M^{(1)}(k;x,t)$ solves the following RHP:

\begin{RHP}  Find a matrix-valued function $M^{(1)}(k;x,t)$ which satisfies
 \begin{itemize}
     \item Analyticity: $M^{(1)}(k;x,t)$ is analytic in $\mathbb{C}\setminus \Sigma^{(1)}$, where $\Sigma^{(1)}= L_1 \cup L_2$.
     \item Jump condition: $M^{(1)}(k;x,t)$ has the jump condition
     $$M^{(1)}_-(k;x,t)=M^{(1)}_+(k;x,t)V^{(1)}(k;x,t), \; k \in \Sigma^{(1)},$$
     where
     \begin{equation}
        V^{(1)}(k;x,t)=\begin{cases}
         \left(\begin{array}{cc}
             1 & 0\\
             -r(k)e^{2it g(k)} & 1
         \end{array}\right),\; k \in L_1,\\
         \left(\begin{array}{cc}
            1 & -\overline{r(k)}     e^{-2it g(k)}\\
           0 & 1
        \end{array}\right),\; k \in L_2.
 \end{cases}
 \end{equation}
     \item Asymptotic behavior:
     \begin{align*}
             M^{(1)}(k;x,t)=I+\mathcal{O}(k^{-1}),	\quad  k \to  \infty.
     \end{align*}
 \end{itemize}
\end{RHP}
It is easy to prove that $M^{(1)}(k;x,t)$ has the following unique solution
\begin{align*}
    M^{(1)}(k;x,t) = I + \frac{1}{2\pi i } \int_{L_1 \cup L_2} \frac{\mu(s)\left( V^{(1)}(s)- I\right)}{s-k} \mathrm{d}s,
\end{align*}
where $\left( I- C_w\right)\mu = I$.  We assume that as $k \to \infty$, $M^{(1)}(k;x,t)= I +\frac{M^{(1)}_1(x,t)}{k} + O(k^{-2})$, then we can find
\begin{align*}
   & M^{(1)}_1 = -\frac{1}{2\pi i }\int_{L_1 \cup L_2} \mu(s) \left(V^{(1)}(s)-I\right) \mathrm{d}s \\
  &  = -\frac{1}{2\pi i }\int_{L_1 \cup L_2}  \left(V^{(1)}(s)-I\right) \mathrm{d}s -\frac{1}{2\pi i }\int_{L_1 \cup L_2} \left(\mu(s)-I \right) \left(V^{(1)}(s)-I\right) \mathrm{d}s.
\end{align*}
For convenience, we take the following special line as an example, other cases can be proved in a similar way.
Define $k_2$ as the  real part of the intersection of the tangent line of $ \im \theta'(k)=0 $ through the point $\left(-\frac{B}{2},  \sqrt{3 \left(\xi + \frac{B^2}{4} \right)}\right)$ with the line $\{k:k=k_1+i \sqrt{6\xi},\;  k_1 \in \mathbb{R} \}$.
Let \begin{align*}
    L_1 = L_1^L + L_1^M + L_1^R
\end{align*}
where
$$L_1^L = \left\{k: k = k_1 + i \sqrt{6\xi}, \; k_1 <k_3 \right\}, \ \  L^R_1 =\left\{ k_1 + i \sqrt{3 \left(\xi + \frac{B^2}{4} \right)}, \; k_1 > -\frac{B}{2} \right\}$$
 and $L_1^M$  is the line connecting $(k_3, i \sqrt{6\xi})$ and $E$ with $k_2<k_3<E $. See Figure \ref{rj6}.
Then we have
\begin{align*}
   \left| \frac{1}{2\pi i }\int_{L_1^L} \left( V^{(1)}(s)-I  \right)\mathrm{d}s  \right| &< \left| \frac{1}{2\pi}  e^{f(\sqrt{6\xi})}  \int_{L_1^L}  e^{-12\sqrt{6\xi}\re^2s}\mathrm{d}s \right| \lesssim  t^{-\frac{1}{2}} e^{12 \sqrt{6} t\xi^{\frac{3}{2}} } \\
   \left| \frac{1}{2\pi i }\int_{L_1^R} \left( V^{(1)}(s)-I  \right)\mathrm{d}s  \right| &< \left| \frac{1}{2\pi} e^{f\left(\sqrt{3\left( \xi +\frac{B^2}{4}\right)} \right)} \int_{L_1^R}  e^{\left(12\sqrt{3\left( \xi +\frac{B^2}{4} \right)}t \right) \re^2s}\mathrm{d}s \right|\\
   &\lesssim t^{-\frac{1}{2}}  e^{12 \sqrt{3 \left(\xi +\frac{B^2}{4} \right)} \left( \xi +\frac{B^2}{4}+1\right) t},
\end{align*}
where $f(x)=4t(x^3 - 3\xi x)$.
Similarly,  \begin{align*}
    \left| \frac{1}{2\pi i }\int_{L_1^M} \left( V^{(1)}(s)-I  \right)\mathrm{d}s  \right| &\lesssim  \left|\int_{k_2}^{-\frac{B}{2}}    \int_{\sqrt{6\xi}}^{\sqrt{3 \left( \xi + \frac{B^2}{4} \right)}} e^{f(y)-12 ty x^2} \mathrm{d} x  \mathrm{d} y  \right| \\
    &\lesssim t^{-\frac{1}{2}}   e^{12 \sqrt{3 \left(\xi +\frac{B^2}{4} \right)} \left( \xi +\frac{B^2}{4}+1\right) t}.
\end{align*}
The second integral $\frac{1}{2\pi i }\int_{L_1 \cup L_2} \left(\mu(s)-I \right) \left(V^{(1)}(s)-I\right) \mathrm{d}s$ can be estimated by similar method.
Also we can prove the cases for $ \frac{A^2}{6}<\xi<A^2$ and $\xi > A^2$ in a similar way, which give us the asymptotic behavior in the slow decay region.

\end{document}